\documentclass[12pt]{amsart}

\usepackage[mathscr]{eucal}
\usepackage{amscd}
\usepackage{amsfonts}
\usepackage{amsmath, amsthm, amssymb}
\usepackage[hidelinks]{hyperref}
\usepackage{latexsym}
\usepackage{graphics}
\usepackage{booktabs,longtable,needspace}

\title{On the $\theta$-split Side of the Local Relative Trace Formula}
\author{Jonathan Sparling}
\date{September 23, 2026}

\newtheorem{thm}{Theorem}[section]
\newtheorem{lemma}[thm]{Lemma}
\newtheorem{coro}[thm]{Corollary}
\newtheorem{prop}[thm]{Proposition}

\theoremstyle{definition}
\newtheorem{definition}[thm]{Definition}

\theoremstyle{remark}

\newtheorem{example}[thm]{Example}

\newcommand{\rank}{\text{\rm rank\,}}

\begin{document}

\begin{abstract}
    The author derives an expression for the $\theta$-split side of the local trace formula at
    the level of Lie algebras, by combining methods of Arthur and Harish-Chandra with the
    structure theory of reductive symmetric spaces.
\end{abstract}

\maketitle

\section{Introduction} \label{1}

The local trace formula of Arthur, derived in \cite{Art91}, is a tool in local harmonic analysis
that identifies a sum of distributions involving (weighted) orbital integrals with a sum of
distributions involving (weighted) characters. This allows one to use harmonic analysis to
better understand the representation theory of a $p$-adic group $H$. Specifically, one obtains
information about the representation $R$ of $H \times H$ on the space $C_c^{\infty}(H)$ of
locally constant and compactly supported complex-valued functions on $H$ that is given by
\[
    (R(h, g) \phi)(x) = \phi(h^{-1}xg)
\]
To derive this formula, Arthur takes two test functions $f_1$ and $f_2$ in $C_c^{\infty}(H)$,
and expresses the averaged operator $R(f_1, f_2)$ as an integral operator with kernel
\[
    K(x, y) = \int_H f_1(h) \, f_2(x^{-1}hy) \, dh
\]
After a suitable truncation procedure, he integrates this kernel along the diagonal to obtain a
(geometric) expression for a modification of the trace of $R(f_1, f_2)$. He then combines this
truncation procedure with the Plancherel formula of Harish-Chandra to obtain a spectral
expression for this modification. The local trace formula is the statement that these two
expressions are equal, which provides a connection between the geometry of $H$ and its
representation theory.

We would like to express Arthur's development in the context of symmetric spaces. The diagonal
embedding of $H$ into $H \times H$ defines a reductive symmetric space $H \backslash H \times H$
and an isomorphism of $C^{\infty}_c(H \backslash H \times H)$ with $C^{\infty}_c(H)$ by
associating to $\phi$ in $C^{\infty}_c(H)$ the function
\[
    (x_1, x_2) \mapsto \phi(x_1^{-1}x_2)
\]
on $H \backslash H \times H$. This isomorphism intertwines $R$ with the right regular
representation of $H \times H$ on $C^{\infty}_c(H \backslash H \times H)$, and we can interpret
$K(x, y)$ as the kernel of this representation. Optimistically, one hopes that Arthur's
development may be generalized to address the right regular representation of a $p$-adic group
$G$ on $C_c^{\infty}(H \backslash G)$ for a large class of reductive symmetric spaces
$H \backslash G$.

The motivating article for pursuing the local trace formula at the level of Lie algebras is
\cite{Wal95}, in which Waldspurger proves that
\begin{align*}
    J(f_1, f_2) & = J(\hat{f_1}, \check{f_2})
\end{align*}
for a bilinear distribution $J(f_1, f_2)$ that equals a sum of certain (orbital) integrals over
the conjugacy classes of maximal tori in $H_F$. He derives this identity from Arthur's local
trace formula with the exponential map, and uses it to prove several substantial results from
local harmonic analysis, including the representability by a specific function of certain
invariant distributions on $\mathfrak{g}$, extending earlier work of Harish-Chandra in
\cite{HC70}. In addition to being a powerful tool in harmonic analysis, this formula also
provides a technically simpler analog of Arthur's local trace formula that nevertheless exhibits
much of its structure. For example, one encounters in it weighted orbital integrals on
$\mathfrak{g}$ with essentially the same weights. For these reasons, one might very much like to
generalize this identity from reductive groups to reductive symmetric spaces.

To discuss this generalization in more detail, we require some notation. We will be working over
a $p$-adic field $F$ with ring of integers $\mathcal{O}$, and a chosen uniformizing element
$\varpi$. We will assume further that $F$ is of characteristic 0 and of residue characteristic
greater than 2. Let $\mathbb{G}$ be a reductive algebraic group defined over $\mathcal{O}$,
which we will assume to be split and connected, and let
$\theta : \mathbb{G} \rightarrow \mathbb{G}$ be an involution on $\mathbb{G}$, which we will
also assume to be defined over $\mathcal{O}$. Write $\mathbb{H} := \mathbb{G}^{\theta}$ for the
fixed points of $\theta$, with $H^\circ$ as its (algebraic) identity component. Let the
$F$-points of $\mathbb{G}$ and $\mathbb{H}$ be written $G_F$ and $H_F$ respectively. The
quotient of groups $H_F \backslash G_F$ has the structure of a reductive symmetric space. We
will occasionally omit the subscript $F$ when considering $F$-points, following the convention
that any quotient written as $H \backslash G$ will be $H_F \backslash G_F$. For a torus $T$, we
will write $N_G(T)$ for the normalizer of $T$, $Z_G(T)$ for its centralizer, and $W_G(T)$ for
its Weyl group. If $\mathfrak{g}$ denotes the Lie algebra of $G_F$ and $\mathfrak{h}$ the Lie
algebra of $H_F$, put
\[
    \mathfrak q:=\{X\in\mathfrak g:\theta(X)=-X\}
\]
Since the residue characteristic is not $2$, one has $\mathfrak g=\mathfrak h\oplus\mathfrak q$.
We choose below a nondegenerate invariant form for which these summands are orthogonal so that
$\mathfrak h^\perp=\mathfrak q$. This is the tangent space to the symmetric variety
$H_F\backslash G_F$ at the identity coset. For simplicity, we will assume that $G_F$ contains a
$(\theta, F)$-split maximal torus, but prove results in greater generality whenever convenient. 
Put $K:=\mathbb G(\mathcal O)$.

We begin the development of the relative local trace formula with a modified version of the
infinitesimal Plancherel identity:
\[
    \int_{\mathfrak{h}^{\perp}} f(g^{-1}Xg) \, dX = \int_{\mathfrak{h}} \hat{f}(g^{-1}Xg) \, dX
\]
This already incorporates Arthur's application of Harish-Chandra's Plancherel formula. What
remains is to truncate both sides of this identity suitably, integrate over the symmetric space
$H \backslash G$, and refine the resulting expressions into two equal distributions of possibly
weighted orbital integrals and characters. In other words, our task will be to study the
identity
\[
    \int_{H \backslash G} \phi(g) \int_{\mathfrak{h}^{\perp}} f(g^{-1}Xg) \, dX\, dg = \int_{H \backslash G}
        \phi(g) \int_{\mathfrak{h}} \hat{f}(g^{-1}Xg) \, dX\, dg
\]
for a carefully chosen function $\phi = \bar\omega(\cdot, \mu) \in C_c^{\infty}(H \backslash G)$. We set
\[
    \begin{aligned}
        J_+(f, \mu) = J_+(f,\omega, \mu) & =\int_{H_F\backslash G_F}\bar\omega(g,\mu)
        \int_{\mathfrak h}f(g^{-1}Xg)\,dX\,d\dot g                                    \\
        J_-(f, \mu) = J_-(f,\omega, \mu) & =\int_{H_F\backslash G_F}\bar\omega(g,\mu)
        \int_{\mathfrak h^\perp}f(g^{-1}Xg)\,dX\,d\dot g
    \end{aligned}
\]

In this paper, we attend to the left-hand side, which we call the $\theta$-split side, and give
an essentially complete formulation for the corresponding part of the trace formula. This
formulation will involve only weighted orbital integrals whose weights are only slightly
modified from those of Arthur.

\subsection{Outline}

In section \hyperref[2]{2}, we review some of the structure theory of reductive symmetric
spaces, and derive the formal expansion of the $\theta$-split side of the relative local trace
formula. We derive a Weyl integration formula for the vector space $\mathfrak{h}^{\perp}$ and
prove that the $\theta$-split side equals
\begin{align*}
     & \sum_{M \in \mathcal{L}^{-}} \sum_{T \in \mathcal{T}_M} \frac{1}{|W_G(A_M)| \cdot |W_M(T)|}
         \int_{\mathfrak{t}'} \left| D^G(X) \right|^{\frac{1}{2}} \int_{(A_M^- \cdot T[2]) \backslash G}
         f(g^{-1}Xg) \, \omega_{M,T} (g, \mu) \, dg\, dX
\end{align*}
where $\mathcal{L}^{-}$ is a set of representatives of certain $G_F$-conjugacy classes of Levi
subgroups, $\mathcal{T}_M$ comprises elliptic maximal tori in $M_F$ whose $G_F$-conjugacy class
contains a $\theta$-split torus, $T[2]:=\ker(t\mapsto t^2)$, and
\begin{align*}
    \omega_{M,T}(g, \mu) & := \sum_{S \in \mathcal{T}_G(T)} \frac{|W_G(S)|}{|W_H(S)|}
        \int_{T[2]_F\backslash (A_M^- \cdot T[2])_F} \phi(g_Stg) \; dt
\end{align*}
The set $\mathcal{T}_G(T)$ is a set of representatives of the $H_F$-conjugacy classes of
$\theta$-split tori contained in the $G_F$-conjugacy class of $T$. In specific cases, the set
$\mathcal{T}_G(T)$ appears to be rather difficult to compute explicitly.

In section \hyperref[3]{3}, we choose a truncation function that
generalizes Arthur's truncation procedure. We fix a particular $F$-split maximal torus $A$ and
let $\phi$ equal the characteristic function of the set
\begin{align*}
    \{ g :  \text{Cartan\,}(\theta(g)^{-1}g) \in \text{Hull} \left\{ W_G(A) \cdot \mu \right\}^* \}
\end{align*}
for some dominant coweight $\mu$ of $A$. The characteristic function of this set will be written
$\bar{\omega}(g,\mu)$.

When $\mu$ is sufficiently regular, we will be able to integrate this truncation function to
obtain weighted lattice sums, which we will ultimately express as finite-character combinations
of equivariant Euler--Poincar\'{e} characteristics of line bundles on a certain toric variety.

Finally, after proving that the integrals in the formal expansion of the $\theta$-split side are
sufficiently well behaved in section \hyperref[4]{4}, we apply the Lebesgue dominated
convergence theorem to prove the main result of this paper, which is the next proposition,
proven in \hyperref[5]{5}.

\begin{prop} \label{main_result}
    Let $\mu_1$ be a dominant coweight and let $\mu_2$ be a strictly dominant coweight in
    $\Lambda^0$, so that $\langle\alpha,\mu_2\rangle>0$ for every simple root $\alpha$. After
    replacing $\mu_2$ by a positive multiple, set $\mu:=\mu_1+d\mu_2$ for
    $d\in\mathbb Z_{\geq0}$. Then
    \begin{align*}
         & \lim_{\substack{d\rightarrow\infty \\ d\in\mathbb Z}}
        \bigl(J_-(f,\omega_\mu)-J_-(f,\nu_\mu)\bigr)=0
    \end{align*}
    The limit is uniform on the supremum-norm unit ball of every finite-dimensional space of
    test functions.
\end{prop}

This approximation is proved independently of the fixed-side paper. That paper constructs
$J_+(f,\nu_\mu)$ and proves the analogous limit. Since any two polynomials whose difference
tends to zero coincide, $J_+(\widehat f,\nu_\mu)=J_-(f,\nu_\mu)$. 

\subsection{Acknowledgements}

The author would like to thank Professor Robert E. Kottwitz for his exceptional support and
guidance, and is particularly grateful for many enlightening conversations and for his
well-placed encouragement. 

\setcounter{tocdepth}{1}
\needspace{18\baselineskip}
\medskip\begingroup\small\tableofcontents\endgroup\medskip

\section{Preliminary Structure Theory} \label{2}

To begin, we will review some definitions and theorems concerning reductive symmetric spaces
that may be found in \cite{Vus74} for algebraically closed fields, and in various papers by
Helminck (e.g., \cite{Hel91}, \cite{Hel97}, \cite{Hel00}) more generally. We consider $p$-adic
fields of characteristic zero and residue characteristic greater than 2, as well as their
algebraic closures.

\subsection{A few Choices and Conventions}

Write
\[
    \mathfrak g=\mathfrak z_{\mathfrak g}\oplus[\mathfrak g,\mathfrak g]
\]
and fix a symmetric bilinear form
\[
    B:\mathfrak g\times\mathfrak g\longrightarrow F
\]
that is nondegenerate, $G$-invariant, and $\theta$-invariant, and whose restriction to every
relevant Cartan subspace is nondegenerate. For example, we could take the Killing form on
$[\mathfrak g,\mathfrak g]$ and choose a sufficiently general nondegenerate $\theta$-invariant
form on the center. With this choice,
\[
    \mathfrak g=\mathfrak h\mathbin{\oplus}\mathfrak h^\perp
\]
and the restriction of $B$ to either summand is nondegenerate.

Fix a nontrivial additive character $\psi:F\to\mathbb C^\times$. On any subspace $V$ for which
$B|_V$ is nondegenerate, we use the unique measure self-dual with respect to $\psi\circ B|_V$.
Set
\[
    \widehat f(Y):=\int_{\mathfrak g}f(X)\psi(B(X,Y))\,dX
\]
The decomposition of $\mathfrak g$ identifies the
self-dual measure on $\mathfrak g$ with the product of those on $\mathfrak h$ and
$\mathfrak h^\perp$. As a result, Fourier inversion on $\mathfrak h$ gives
\[
    \int_{\mathfrak h}\widehat f(X)\,dX
    =\int_{\mathfrak h^\perp}f(Y)\,dY
\]
Because $B$ is $G$-invariant, the Fourier transform commutes with the adjoint action.

More broadly, when a reductive group's Lie algebra carries a nondegenerate restriction of $B$,
we choose the invariant Haar measure whose (tangent) density at the identity is self-dual. For
finite groups, we use a counting measure with unit mass per point, and this gives a Haar measure
on $H_F^\circ$ from the Haar measure on $H_F$. All quotient measures are defined implicitly by
the integration-by-stages formula. As an example, if $T$ is a maximal $\theta$-split torus, the tangent
space of $Z_H(T)_F\backslash H_F$ is then the $B$-orthogonal complement of
$Z_{\mathfrak h}(\mathfrak t)$ in $\mathfrak h$; its quotient measure has the self-dual tangent
density there.

\subsection{\texorpdfstring{Tori Adapted to $\theta$}{Tori Adapted to theta}}

Much of the structure theory of $H_F \backslash G_F$ is encoded in the relationship between
$\theta$ and the structure theory of $G_F$. Because there always exists a torus preserved by the
action of $\theta$ (e.g., \cite{Hel91}), we can study the action of $\theta$ on its roots and
weights. One complication, however, is that this action can vary from one torus to the next. The
next definition, in part, identifies those tori on which the action of $\theta$ is particularly
simple.

\begin{definition}
    Suppose that $T \subset G$ is a torus. Then $T$ is $\theta$-stable if $\theta(T) \subset T$;
    $T$ is $\theta$-split if $\theta(t) = t^{-1}$ for all $t \in T$; $T$ is $\theta$-fixed if
    $\theta(t) = t$ for all $t \in T$; and $T$ is $(\theta, F)$-split if $T$ is both
    $\theta$-split and $F$-split.
\end{definition}

Suppose that $T$ is a $\theta$-stable torus. Let $X^{*}(T)$ denote the group of rational
characters of $T$, and let $X_{*}(T)$ denote the group of rational cocharacters of $T$. $\theta$
then acts by composition on $X_{*}(T)$ and by precomposition on $X^{*}(T)$. While a
$\theta$-stable torus $T$ will not usually be either $\theta$-split or $\theta$-fixed, it will
certainly contain $\theta$-fixed and $\theta$-split subtori. Set
\begin{align*}
    T^- & := \{ t \in T : \theta(t) = t^{-1} \}^0 \text{ and } T^+ := \{ t \in T : \theta(t) = t \}^0
\end{align*}
Geometrically, the product map
\begin{align*}
    \mu : T^+ \times T^- \rightarrow T \hspace{3ex} (t_1, t_2) \mapsto t_1 \cdot t_2
\end{align*}
is an isogeny. It provides a decomposition $T = T^+ \cdot T^-$ where $T^- \cap T^+$ is a finite
group. In fact, this intersection has the form $(\mathbb{Z} / 2\mathbb{Z})^n$ for some $n$.

The possibly restricted abstract root systems of $\mathfrak{t}$ and $T$ in $G$ will be written
$\Phi(\mathfrak{t}, G)$ and $\Phi(T, G)$ respectively. We will also need to generalize the
notion of a Cartan subalgebra from Lie algebras to this symmetric space setting.

\begin{definition}
    A Cartan subspace of $\mathfrak{h}^{\perp}$ is a maximal abelian subspace of
    $\mathfrak{h}^{\perp}$ that contains only semisimple elements.
\end{definition}

It follows from this definition that the Lie algebra of a maximal $\theta$-split torus is a
Cartan subspace of $\mathfrak{h}^{\perp}$, and that every such Cartan subspace arises in this
way. We will be interested at first in the $H_F$-conjugacy classes of Cartan subspaces of
$\mathfrak{h}^{\perp}$, or, equivalently, the $H_F$-conjugacy classes of maximal $\theta$-split
tori. Over $\bar{F}$, these collapse into a single conjugacy class.

\begin{thm} \cite{Vus74}
    $H_{\bar{F}}$ acts transitively on the maximal $\theta$-split tori of $G_{\bar{F}}$.
\end{thm}

One consequence of this theorem is that all Cartan subspaces of $\mathfrak{h}^{\perp}$ have the
same dimension, which is then a uniquely defined number.

\begin{definition}
    The rank of $\mathfrak{h}^{\perp}$ is the dimension of any Cartan subspace in
    $\mathfrak{h}^{\perp}$.
\end{definition}

There is also a notion of regularity adapted to these Cartan subspaces, which is not equivalent
to regularity in $\mathfrak{g}$ for general symmetric spaces.

 \begin{definition}
    We say that an element $X$ in $\mathfrak{h}^{\perp}$ is $\theta$-regular if the centralizer
    $\mathfrak{h}^{\perp}_X$ has minimal dimension among all centralizers of elements in
    $\mathfrak{h}^{\perp}$. Otherwise, we say that $X$ is $\theta$-singular.
\end{definition}

For example, if $\theta$ is the trivial involution, then $\mathfrak{h}^{\perp}$ is trivial, and
0 is $\theta$-regular, but not regular in $\mathfrak{g}$ unless $\mathfrak{g}$ is abelian. The
expression $\mathfrak{h}^{\perp}_{\theta-reg}$ will denote the set of elements in
$\mathfrak{h}^{\perp}$ that are $\theta$-regular. Often, we will also consider the centralizers
of maximal $\theta$-fixed and $\theta$-split tori, which are described in the following lemma.

\begin{lemma} \label{centralizer_lemma} \cite{Vus74}
    If $T$ is a maximal torus in $H$, then $Z_G(T)$ is a maximal torus in $G$. If $T$ is a
    maximal $\theta$-split torus in $G$ then $[Z_G(T), Z_G(T)] \subset H$.
\end{lemma}

From the theory of algebraic groups, applied to the subgroup $H_F$, we know that any two maximal
$\theta$-fixed tori that are also $F$-split are conjugate by an element of $H_F$. The same,
however, is not necessarily true for $(\theta, F)$-split tori. In fact, even $SL_2$ provides a
counterexample.

\begin{example} \label{example_one}
    Take $\mathfrak{g} = \mathfrak{sl}_2$, and $\theta$ to be the involution
    \begin{align*}
        \theta : \begin{pmatrix} a & b \\ c & -a \end{pmatrix} \mapsto \begin{pmatrix} -1 & 0 \\ 0 & 1
            \end{pmatrix} \begin{pmatrix} a & b \\ c & -a \end{pmatrix} \begin{pmatrix} -1 & 0 \\ 0 & 1
            \end{pmatrix} = \begin{pmatrix} a & -b \\ -c & -a \end{pmatrix}
    \end{align*}
    Then $\mathfrak{h} = \mathfrak{t}$, the subspace of diagonal matrices, and
    \begin{align*}
        \mathfrak{h}^{\perp} = \left\{ \begin{pmatrix} 0 & b \\ c & 0 \end{pmatrix} : b, c \in F \right\}
    \end{align*}
    For every $c \in F^{\times}$, there is a corresponding Cartan subspace of
    $\mathfrak{h}^{\perp}$:
    \begin{align*}
        \mathfrak{a}_c := \left\{ \begin{pmatrix} 0 & ct \\ t & 0 \end{pmatrix} : t \in F \right\}
    \end{align*}
    Two such Cartan subspaces, say $\mathfrak{a}_c$ and $\mathfrak{a}_d$, are conjugate by an
    element of $H_F$ if and only if they are conjugate by a diagonal matrix, which implies that
    $c = d r^4$ for some $r \in F^{\times}$. In other words, they are conjugate precisely when
    $c$ and $d$ both represent the same fourth power class. On the other hand,
    \begin{align*}
        \begin{pmatrix} 0 & ct \\ t & 0 \end{pmatrix}
    \end{align*}
    is, for $t\neq0$, $F$-split precisely when its characteristic polynomial splits over $F$, i.e., when $c$ is a
        square. When $F$ is $p$-adic, there will exist squares in $F$ that are not fourth powers, and so we can
        construct two distinct $(\theta, F)$-split tori that are not $H_F$-conjugate.
\end{example}

Some reductive symmetric spaces lack this complication and all maximal $(\theta, F)$-split tori are
    $H_F$-conjugate. For most symmetric spaces, however, two maximal $(\theta, F)$-split tori will only be
    conjugate over a slightly enlarged set.

\begin{thm} \rm (Proposition 10.3 in \cite{HW93}) \it
    Let $A_1$ and $A_2$ be two maximal $(\theta, F)$-split tori, and let $A_F$ be a $F$-split maximal torus such
        that $A_1 \subset A_F \subset G$. Then
    \begin{align*}
        g^{-1}A_1g = A_2 \text{ for some } g \in (Z_G(A)H)_F
    \end{align*}
    where $A$ is the algebraic torus whose $F$-points are $A_F$.
\end{thm}

\begin{coro}
    All maximal $(\theta, F)$-split tori in $G_F$ have the same rank.
\end{coro}

The proof of this theorem is instructive, and we will summarize part of it, as it applies to maximal $\theta$-split
    tori. Suppose that $S$ and $T$ are maximal $\theta$-split tori that are conjugate over $G_F$. Then there exists
    a $g_S$ in $G_F$ such that
\[
    g_S^{-1} t g_S = s \in S \text{ for an arbitrary $t$ in $T$}
\]
By applying $\theta$ to this equation and taking inverses,
\[
    \theta(g_S)^{-1} t \theta(g_S) = s
\]
This then implies that the $\theta$-split element $\theta(g_S)g_S^{-1}$ belongs to the centralizer $Z_G(T)$ of $T$.
    But the derived group of $Z_G(T)$ is a subgroup of $H$, and so
\[
    g_S \in (\tilde{T}H)_F
\]
where $\tilde{T}$ is the center of $Z_G(T)$. One difficulty in the structure theory of reductive symmetric spaces
    is that we cannot say very much more about $g_S$. Certainly $g_S$ does not necessarily belong to $H_F$, and so
    one is usually forced to work with this slightly larger set. We summarize some of these observations as a
    proposition.

\begin{prop}
    If $g_S$ conjugates a maximal $\theta$-split torus $S$ to another maximal $\theta$-split torus $T$, then $g_S$
        is an element of the set $(\tilde{T}H)_F$. In particular, the groups $N_G(T)_F$ and $N_{\tilde{T}H}(T)_F$
        are equal and elements in the Weyl group $W_G(T)$ have representatives in $(\tilde{T}H)_F$.
\end{prop}

Suppose that we are given a maximal $\theta$-split torus $T$. From a representative $g_S$ of a double coset in
    $N_{\tilde{T}H}(T)_F \backslash (\tilde{T}H)_F / H_F$, we can define a torus $S := g_S^{-1} T g_S$ and an
    isomorphism
\begin{align*}
    T \rightarrow S \qquad t \mapsto g_S^{-1}tg_S
\end{align*}
Notice that this isomorphism commutes with the involution $\theta$ and that there is an isomorphism of restricted
    root systems:
\begin{align*}
    \Phi(T, G) \cong \Phi(S, G)
\end{align*}
In addition, the next proposition shows that the $H_F$-conjugacy classes of maximal $\theta$-split tori in $G_F$
    that are conjugate to $T$ are parametrized by elements $g_S$ that precisely represent the double cosets
    $N_{\tilde{T}H}(T)_F \backslash (\tilde{T}H)_F / H_F$. We will need this result in section \hyperref[3]{3}.

\begin{prop}
    The $H_F$-conjugacy classes of maximal $\theta$-split tori that are $G_F$-conjugate to $T$ are parametrized by
        elements of $N_{\tilde{T}H}(T)_F \backslash (\tilde{T}H)_F / H_F$:
    \begin{align*}
        K_{\theta}(F, T) := \left\{ \text{ maximal }
            \theta\text{-split tori that are $G_F$-conjugate to $T$}\right\} & / \left( H_F-\text{conjugacy}\right)
            \\ \leftrightarrow N_{\tilde{T}H}(T)_F \backslash& (\tilde{T}H)_F / H_F
    \end{align*}
\end{prop}

\begin{proof}
    This bijection is given explicitly by the map
    \[
        g_S \mapsto S
    \]
    We have seen that this map is surjective, so it remains to check that it is also injective. To wit, if $S = S'$
        for representatives $g_S$ and $g_{S'}$ of different double cosets, then $x := g_{S'} g_S^{-1} \in
        (\tilde{T}H\tilde{T})_F$ normalizes $T$. But such an $x$ is equal to $t_1ht_2$ for $t_1, t_2 \in
        \tilde{T}_{\bar{F}}$ and $h \in H_{\bar{F}}$, and therefore $h$ must normalize $T_{\bar{F}}$ and
        $\tilde{T}_{\bar{F}}$. This implies that $x = t_1 h t_2 h^{-1} h \in (\tilde{T}H)_F$. Since $x$ also
        normalizes $T$, we have
    $x\in N_{\tilde{T}H}(T)_F$, which implies that $g_S$ and $g_{S'}$ both represent the same double coset of
        $N_{\tilde{T}H}(T)_F \backslash (\tilde{T}H)_F / H_F$, proving injectivity.
\end{proof}

There is another interpretation of these double cosets that is perhaps computationally simpler. The map $\tau_R : g
    \mapsto \theta(g)g^{-1}$ sends $N_{\tilde{T}H}(T)_F \backslash (\tilde{T}H)_F / H_F$ to a set of
    $\theta$-twisted $N_{\tilde{T}H}(T)_F$-conjugacy classes in $\tau_R((\tilde{T}H)_F) \subset T_F$. This
    perspective provides some intuitive justification that $K_{\theta}(F, T)$ is finite. A proof may be found in,
    e.g., \cite{Hel97}.

\subsection{\texorpdfstring{A Weyl Integration Formula for $\mathfrak{h}^{\perp}$}{A Weyl Integration Formula for the Anti-fixed Space}}

In this section, we will prove an analog of the Weyl integration formula for $\mathfrak{h}^{\perp}$, loosely
    following part of \cite{RR96}. This differs from the development of the formula for Lie algebras mainly in the
    computation of a Jacobian, which requires a couple of algebraic tricks. We present this computation first.

\begin{prop}
    Let $T$ be a maximal $\theta$-split torus of $G$ with Lie algebra $\mathfrak{t}$. Let $U = U(\mathfrak{t})$ be
        the set of points in $\mathfrak{h}^{\perp}_{\theta-reg}$ whose $H_F$-orbit intersects $\mathfrak{t}$.

    \begin{itemize}
        \item Set $\mathfrak{t}' := \mathfrak{h}^{\perp}_{\theta-reg} \cap \mathfrak{t}$.  The map
              \begin{align*}
                  \phi_{\mathfrak{t}} : \mathfrak{t}' \times Z_H(T) \backslash H \rightarrow U \hspace{3ex} (A, h)
                      \mapsto h^{-1}Ah
              \end{align*}
              is $|W_H(T)|$-to-one, regular, and surjective.

        \item Its Jacobian is
              \begin{align*}
                  \left| \prod_{\alpha \in \Phi(\mathfrak{t}, G)} \alpha(A)^{m_{\alpha}} \right|^{\frac{1}{2}}
              \end{align*}
              where $m_{\alpha}$ is the number of roots of some maximal Cartan subalgebra containing $\mathfrak{t}$
                  whose restriction to $\mathfrak{t}$ is $\alpha$.
    \end{itemize}
\end{prop}

\begin{proof}
    First, we need to show that the map $\phi_{\mathfrak{t}}$ is well-defined. Indeed, for any $z \in Z_H(T)$ and
        any $h \in H$,
    \begin{align*}
        \phi_{\mathfrak{t}} (A, zh) & = h^{-1} z^{-1} A z h = h^{-1} A h = \phi_{\mathfrak{t}} (A, h)
    \end{align*}
    so this map is defined on the set $ \mathfrak{t}' \times Z_H(T) \backslash H$.

    Next, we determine the extent to which injectivity can fail.  Suppose that
    \begin{align*}
        \phi_{\mathfrak{t}}(A_1, h_1) & = \phi_{\mathfrak{t}}(A_2, h_2)
    \end{align*}
    Then $h_1^{-1} A_1 h_1 = h_2^{-1} A_2 h_2$, which means that $h_2 h_1^{-1} A_1 h_1 h_2^{-1} = A_2$. Since $A_1$
        and $A_2$ are both elements of $\mathfrak{t}'$,
    \begin{align*}
        h_2 h_1^{-1} \in N_H(T)
    \end{align*}
    $N_H(T)$ therefore acts transitively on the preimage of $h_1^{-1}A_1h_1$ and the stabilizer of any element
        under this action is $Z_H(T)$. By the orbit-stabilizer theorem, choosing an element in the preimage of
        $\phi_{\mathfrak{t}} (A_1, h_1)$ defines a 1-1 correspondence between this preimage and $W_H(T)$. Since $U$
        was chosen so that $\phi_{\mathfrak{t}}$ would be surjective, to complete the proof of this proposition, we
        only need to compute the Jacobian.

    Because $\phi_{\mathfrak t}$ is $H_F$-equivariant and all the measures
    are invariant, it is enough to compute the derivative at $(A,1)$.  First,
    splinter off the centralizer of $\mathfrak t$ from $\mathfrak h$:
    \[
        \mathfrak h=\mathfrak h_0\oplus\mathfrak h_1\qquad
        \mathfrak h_0:=Z_{\mathfrak h}(\mathfrak t)\qquad
        \mathfrak h_1:=\mathfrak h_0^{\perp_B}\cap\mathfrak h
    \]
    Similarly, splinter off $\mathfrak t$ from $\mathfrak h^\perp$:
    \[
        \mathfrak h^\perp=\mathfrak t\oplus\mathfrak q_1\qquad
        \mathfrak t = Z_{\mathfrak h^\perp}(\mathfrak t)\qquad
        \mathfrak q_1 :=\mathfrak t^{\perp_B}\cap\mathfrak h^\perp
    \]
    Write $\mathfrak t'=\mathfrak t_{\rm reg}$. Then for $A\in\mathfrak t'$, the maps
    \[
        A_{\mathfrak h}:\mathfrak h_1\longrightarrow\mathfrak q_1
        \quad U\longmapsto[A,U]
        \qquad
        A_{\mathfrak q}:\mathfrak q_1\longrightarrow\mathfrak h_1
        \quad V\longmapsto[A,V]
    \]
    are isomorphisms.  Invariance and symmetry of $B$ give
    \[
        B(A_{\mathfrak h}U,V)=-B(U,A_{\mathfrak q}V)
    \]
    which means that $A_{\mathfrak q}$ is the negative adjoint of
    $A_{\mathfrak h}$, and so their
    determinants therefore have the same absolute value.

    The derivative of $\phi_{\mathfrak t}$ is
    \[
        d(\phi_{\mathfrak t})_{(A,1)}:
        \mathfrak t\oplus\mathfrak h_1\longrightarrow
        \mathfrak t\oplus\mathfrak q_1\qquad
        (Y,U)\longmapsto Y+A_{\mathfrak h}U
    \]
    which has Jacobian $|\det A_{\mathfrak h}|$.  As well,
    $\operatorname{ad}(A)$ acts on
    \[
        \mathfrak g/Z_{\mathfrak g}(A)
        =\mathfrak h_1\oplus\mathfrak q_1
    \]
    with off-diagonal blocks $A_{\mathfrak h}$ and
    $A_{\mathfrak q}$.  Consequently
    \[
        \left|\det\bigl(\operatorname{ad}(A);
        \mathfrak g/Z_{\mathfrak g}(A)\bigr)\right|
        =|\det A_{\mathfrak h}|^2
    \]
    After extension to $\bar F$, the determinant on the left is
    \[
        \left|\prod_{\alpha\in\Phi(\mathfrak t,G)}
        \alpha(A)^{m_\alpha}\right|
    \]
    and so
    \[
        \operatorname{Jac}(\phi_{\mathfrak t})(A,h)
        =\left|\prod_{\alpha\in\Phi(\mathfrak t,G)}
        \alpha(A)^{m_\alpha}\right|^{\frac12}
    \]
    which completes the proof.
\end{proof}

Note that rescaling the measure on $\mathfrak t$ and the measure on
$Z_H(T)\backslash H$ by reciprocal constants leaves their product (and
eventually our Weyl integration formula) unchanged.  Note also that
$|W_H(T)|$ might depend on the field $F$.
For example, in Example~\ref{example_one}, this constant is half the
number of fourth roots of unity in $F$ when $T$ is maximal
$(\theta,F)$-split.

We also extract two observations from the proof of this proposition. First, if $\mathfrak{t} \subset
    \mathfrak{h}^{\perp}$ is a Cartan subalgebra, and $A \in \mathfrak{t}$ is semisimple, then for $r=\dim Z_G(T)$,
\begin{align*}
    \det(\operatorname{ad}(A) + \lambda I ; \mathfrak{g}) & = \left( \prod_{\alpha \in \Phi(\mathfrak{t}, G)}
        \alpha(A)^{m_{\alpha}} \right) \lambda^r + \cdots
\end{align*}
So $A \in \mathfrak{t}'$ if and only if the coefficient of $\lambda^r$ does not vanish. Second, we can break up the
    Zariski-open set of all $\theta$-regular semisimple elements of $\mathfrak{h}^{\perp}$ into a disjoint union of
    sets of the form $U(\mathfrak{t})$:
\begin{align*}
    \coprod_{\mathfrak{t} \subset \mathfrak{h}^{\perp}} U(\mathfrak{t}) =
        \coprod_{\mathfrak{t} \subset \mathfrak{h}^{\perp}} \text{Ad}(H) (\mathfrak{t}')
\end{align*}
Here the disjoint sum runs over a set of representatives of the $H_F$-conjugacy classes of Cartan subalgebras
    $\mathfrak{t}$ in $\mathfrak{h}^{\perp}$.

\begin{thm}
    Let $f \in C_c^{\infty}(\mathfrak{g})$.   Then
    \begin{align*}
        \int_{\mathfrak{h}^{\perp}} f(X)\; dX & = \sum_{T \in \mathcal{T}_{G/H}} \frac{1}{|W_H(T)|}
            \int_{\mathfrak{t}'} \left| \prod_{\alpha \in \Phi(\mathfrak{t}, G)} \alpha(X)^{m_{\alpha}}
            \right|^{\frac{1}{2}} \int_{Z_H(T) \backslash H} f (h^{-1}Xh)\; \dot{dh}\; dX
    \end{align*}
    where $\mathcal{T}_{G/H}$ denotes a set of representatives of the
    $H_F$-conjugacy classes of maximal $\theta$-split tori in $G_F$, and
    all measures are the compatible measures fixed above.
\end{thm}

\begin{proof}
    Because the complement of the set of semisimple elements that belong to some $\mathfrak{t}'$ has measure zero
        in $\mathfrak{h}^{\perp}$,
    \begin{align*}
        \int_{\mathfrak{h}^{\perp}} f(X) \; dX & = \sum_{T \in \mathcal{T}_{G/H}} \int_{U(\mathfrak{t})} f(X) \; dX
    \end{align*}
    We can then pull back each summand by the corresponding $\phi_{\mathfrak{t}}$, and this sum becomes
    \begin{align*}
        \sum_{T \in \mathcal{T}_{G/H}} \frac{1}{|W_H(T)|} \int_{\mathfrak{t}'} \left|
            \prod_{\alpha \in \Phi(\mathfrak{t}, G)} \alpha(X)^{m_{\alpha}} \right|^{\frac{1}{2}}
            \int_{Z_H(T) \backslash H} f(h^{-1}Xh) \; \dot{dh} \; dX
    \end{align*}
    by the previous proposition.
\end{proof}

\subsection{\texorpdfstring{A Formal Expression for the $\theta$-split Side}{A Formal Expression for the theta-split Side}}

The next integration formula provides an expression for one side of the trace formula for an unspecified function
    $\phi$ in $C_c^{\infty}(H\backslash G)$. We will derive and study this formula for unspecified weight factors
    for the remainder of this section. In section \hyperref[3]{3}, we will choose and study a particular $\phi$
    that generalizes Arthur's approach to this setting.

\begin{thm} \label{formal_identity}
    Let $f$ be a function in $C_c^{\infty}(\mathfrak{g})$.   Then
    \begin{align*}
        \int_{H \backslash G} & \int_{\mathfrak{h}} \hat{f}(g^{-1}Xg) \, dX\, \phi(g) \, \dot{dg} \\
                              & = \sum_{T \in \mathcal{T}_{G/H}} \frac{1}{|W_H(T)|} \int_{\mathfrak{t}'} \left|
                                  \prod_{\alpha \in \Phi(\mathfrak{t}, G)} \alpha(X)^{m_{\alpha}}
                                  \right|^{\frac{1}{2}} \int_{Z_G(T) \backslash G} f (g^{-1}Xg) \, \omega_T (g) \,
                                  \dot{dg}\; dX
    \end{align*}
    where
    \begin{align*}
        \omega_T(g) & := \int_{Z_H(T) \backslash Z_G(T)} \phi(tg) \; \dot{dt}
    \end{align*}
    for any left $H$-invariant function $\phi \in C_c^{\infty}(H \backslash G)$. The function $\omega_T$ is called
        a weight factor.
\end{thm}

\begin{proof}
    By the $G$-invariance of $B$, the Fourier transform commutes with the
    adjoint action.  The orthogonal decomposition
    $\mathfrak g=\mathfrak h\mathbin{\oplus}\mathfrak h^\perp$ and the
    self-dual measure let us apply Fourier inversion:
    \begin{align*}
        \int_{\mathfrak{h}} \hat{f}(g^{-1}Xg)\; dX & = \int_{\mathfrak{h}^{\perp}} f(g^{-1}Xg)\; dX
    \end{align*}
    Each side of this identity depends as a function of $g$ only on the coset $Hg$, so we can integrate over the
        symmetric space.
    \begin{align}\label{infinitesimal_identity}
        \int_{H\backslash G} \phi(g) \int_{\mathfrak{h}} \hat{f}(g^{-1}Xg) \, dX\; \dot{dg} & =
            \int_{H \backslash G} \phi(g) \int_{\mathfrak{h}^{\perp}} f(g^{-1}Xg) \, dX\; \dot{dg}
    \end{align}
    for any left $H$-invariant $\phi \in C_c^{\infty} (H \backslash G)$, which we introduce to guarantee
        convergence. By the Weyl integration formula for $\mathfrak{h}^{\perp}$, \eqref{infinitesimal_identity}
        equals
    \begin{align*}
         & \int_{H \backslash G} \sum_{T \in \mathcal{T}_{G/H}} \frac{1}{|W_H(T)|} \int_{\mathfrak{t}'} \left|
             \prod_{\alpha \in \Phi(\mathfrak{t}, G)} \alpha(X)^{m_{\alpha}} \right|^{\frac{1}{2}}
             \int_{Z_H(T) \backslash H} f((hg)^{-1}X(hg))\, \phi(g)\, \dot{dh} \, dX\, \dot{dg}
    \end{align*}
    The integration by stages formula for algebraic groups allows us to change variables from $H \backslash G
        \times Z_H(T) \backslash H$ to $Z_H(T) \backslash G$ and from $Z_H(T) \backslash G$ to $Z_G(T) \backslash G
        \times Z_H(T) \backslash Z_G(T)$. By this formula, and a change in the order of integration, we obtain
    \begin{align*}
        \sum_{T \in \mathcal{T}_{G/H}} \frac{1}{|W_H(T)|} \int_{\mathfrak{t}'} \left|
            \prod_{\alpha \in \Phi(\mathfrak{t}, G)} \alpha(X)^{m_{\alpha}} \right|^{\frac{1}{2}}
            \int_{Z_G(T) \backslash G} f(g^{-1}Xg) \left( \int_{Z_H(T) \backslash Z_G(T)} \phi(tg)\, \dot{dt}
            \right) \dot{dg}\, dX
    \end{align*}
    which completes the proof.
\end{proof}
We call the left-hand side of this identity the $\theta$-fixed side of the trace formula. The right-hand side is
    called the $\theta$-split side. The integrals over the $G_F$-orbits of $\theta$-regular semisimple elements in
    $\mathfrak{h}^{\perp}$ will be called $\theta$-split orbital integrals. We set
\[
    D^G_{Z_G(T)} (X) :=   \prod_{\alpha \in \Phi(\mathfrak{t}, G)} \alpha(X)^{m_{\alpha}}
\]
where $\mathfrak{t}$ is a Cartan subalgebra containing $X$ in $\mathfrak{h}^{\perp}$.

\subsection{Preliminary Refinements}

Before we introduce a specific weight factor, we will refine this expression for the $\theta$-split side in three
    ways: first, we simplify the expression for $\omega_T(\cdot)$; second, we refine the sum over the set
    $\mathcal{T}_{G/H}$; and third, we shift the dependence of $\omega_T$ on tori $T$ to Levi subgroups
    parametrizing subsets of $\mathcal{T}_{G/H}$.

\subsubsection{Centralizers of Tori}

The weight factors of theorem \ref{formal_identity} are expressed as integrals over reductive symmetric spaces of
    the form $Z_H(T) \backslash Z_G(T)$, where $T$ is a maximal $\theta$-split torus. In this subsection, we
    consider the image of this quotient under the map
\[
    \tau : H \backslash G \rightarrow G \hspace{5ex} g \mapsto \theta(g)^{-1}g
\]
In this case, the image will be a subgroup of the $F$-points of an abelian group containing a $\theta$-split torus.
    For $\theta$-fixed tori, the analogous statement follows from lemma \ref{centralizer_lemma}.

\begin{lemma} \rm \cite{Vus74} \it
    $Z_H(T) \backslash Z_G(T)$ is a quotient of tori, and $\tau(Z_H(T) \backslash Z_G(T))$ is a subset of the torus
        $Z_G(T)^-$.
\end{lemma}

The weight factor in the statement of theorem \ref{formal_identity} is an integral over $Z_H(T) \backslash Z_G(T)$,
    and the integrand is a function of the image of a coset in this quotient under the map $\tau$.

\begin{lemma}
    Remember that we have written $\tilde{T}$ for the center of $Z_G(T)$. The image of $Z_H(T) \backslash Z_G(T)$
        under the map $\tau$ is an abelian subgroup of the $F$-points of $\tilde{T}$.
\end{lemma}

\begin{proof}
    Over the algebraic closure, any $g \in Z_G(T)$ can be written as a product $sd$ where $s$ belongs to the center
        of $Z_G(T)_{\bar{F}}$ and $d$ belongs to its derived group. By lemma \ref{centralizer_lemma}, the derived
        group $[Z_G(T), Z_G(T)]$ is a subgroup of $H$, and so we can compute $\tau$ at $g$:
    \[
        \tau(Z_H(T) sd) = \tau(Z_H(T)s) = \theta(s)^{-1} s \in \tilde{T}_{\bar{F}}
    \]
    Because
    \[
        Z_H(T)_F \backslash Z_G(T)_F \hookrightarrow (Z_H(T) \backslash Z_G(T))_F
    \]
    is sent by $\tau$ to $\tilde{T}_F$, the result follows.
\end{proof}

Notice that this image need not be the entire set $\tilde{T}_F$. For example, if $G$ is itself a $\theta$-split
    torus, then $\tau$ can only equal squares.
 
\subsubsection{\texorpdfstring{Sums Over $H_F$-conjugacy Classes}{Sums Over H(F)-conjugacy Classes}}

We also wish to refine the sum over $\mathcal{T}_{G/H}$ in theorem \ref{formal_identity} in two ways. First, we
    would like to group elements of $\mathcal{T}_{G/H}$ into $G_F$-conjugacy classes. Second, we would like to make
    the contribution from each conjugacy class of Levi subgroups of $G$ more apparent. To accomplish these goals,
    we require some notation.

Let $x_K$ be the hyperspecial point of the Bruhat--Tits building
associated to $K = \mathbb{G}(\mathcal O)$. By \cite[Proposition~4.8]{DS11}, we can choose a $\theta$-stable
    maximal $F$-split torus $S$
whose apartment contains $x_K$ and with $S^-$ maximal $(\theta, F)$-split.
Since our symmetric space contains a $(\theta, F)$-split maximal torus, and these tori are all $G_F$-conjugate, we
    also have
$\dim S^-=\rank G=\dim S$ and consequently $S^-=S$. We fix such a torus and call it $A$, with $x_K$ as its
    apartment origin.
Then  $A(\mathcal O)=A\cap K$ and by $\alpha$-conversion $A^-=A$. For any Levi subgroup
$M_F$, let $A_M$ denote the $F$-split component of its center, and put
\[
    \mathcal L^-_{\mathrm{all}}
    :=\{M\supset A:M=Z_G(A_M^-)\}
\]
Because $A_M^-=A_M$, this set comprises all Levi subgroups containing
$A$.  Choose $\mathcal L^-\subset\mathcal L^-_{\mathrm{all}}$ to contain
one representative from each $G_F$-conjugacy class.  For each $M$, let
$\mathcal T_M$ represent the $M_F$-conjugacy classes of elliptic maximal
$F$-tori $T\subset M$ whose $G_F$-class contains a $\theta$-split torus.
Let $\mathcal T_G(T)$ be the
$H_F$-conjugacy classes of $\theta$-split tori in that $G_F$-conjugacy class.
Write
\[
    W_G(A_M):=N_G(M)_F/M_F
\]
This is a finite group.

As with the next three lemmas, this notation follows \cite[\S7.12]{Kot06}.
Our goal is to replace the sum over $\mathcal{T}_{G/H}$ with an iterated
sum over $\mathcal{L}^-$, $\mathcal{T}_M$, then $\mathcal{T}_G(T)$.
This requires computing some normalization constants.

\begin{lemma}
    Suppose that $T \in \mathcal{T}_M$ and that $g^{-1}Tg \subset M$ for some $g \in G_F$.   Then $g \in N_G(M)_F$.
\end{lemma}

\begin{proof}
    The torus $A_M$ is central in $M$ and is contained in every maximal
    torus of $M$.  Since $T$ is elliptic in $M$, the split component of
    $g^{-1}Tg$ is $g^{-1}A_Mg$ and has the same dimension as $A_M$.
    It contains $A_M$, hence equals it.  Thus $g$ normalizes $A_M$ and
    $M=Z_G(A_M)$.
\end{proof}

\begin{lemma}
    Let $T\in\mathcal T_M$.  Then the number of $M_F$-conjugacy classes of
    maximal tori in $M$ that are $G_F$-conjugate to $T$ is
    \[
        |W_G(A_M)|\frac{|W_M(T)|}{|W_G(T)|}
    \]
\end{lemma}

\begin{proof}
    The preceding lemma shows that $g\mapsto gTg^{-1}$ maps $M_F\backslash N_G(M)_F/N_G(T)_F$
    bijectively onto this set of conjugacy classes.  In fact, $M_F$ is
    normal in $N_G(M)_F$ and $N_G(T)_F \subset N_G(M)_F$: an element normalizing $T$ normalizes
    its $F$-split component $A_M$ and therefore normalizes
    $M=Z_G(A_M)$.  So
    \[
        |M_F\backslash N_G(M)_F/N_G(T)_F|=\frac{|N_G(M)_F/M_F|}{|N_G(T)_F/(M_F\cap N_G(T)_F)|}
        =\frac{|N_G(M)_F/M_F|}{|N_G(T)_F/N_M(T)_F|}
    \]
    Since $Z_G(T)=T\subset M$,
    \[
        |N_G(T)_F/N_M(T)_F|=\frac{|W_G(T)|}{|W_M(T)|}
    \]
    By definition, $|N_G(M)_F/M_F|=|W_G(A_M)|$, and this completes the proof.
\end{proof}

The $G_F$-conjugacy class of a maximal torus $T$ determines
the conjugacy class of the Levi subgroup $Z_G(A_T)$.
Choosing the elements of $\mathcal L^-$
therefore assigns a unique Levi representative to each such $G_F$-conjugacy class.

\begin{lemma}
    The sum on the $\theta$-split side of the trace formula can be rewritten with these normalization constants:
    \begin{align*}
        \sum_{T \in \mathcal{T}_{G/H}} \frac{1}{|W_H(T)|} \left( \cdot \right) & = \sum_{M \in \mathcal{L}^-}
            \sum_{T \in \mathcal{T}_M} \frac{1}{|W_G(A_M)| \cdot |W_M(T)|} \sum_{S \in \mathcal{T}_G(T)}
            \frac{|W_G(S)|}{|W_H(S)|} \left( \cdot \right)
    \end{align*}
    Here $\left( \cdot \right)$ is shorthand for the rest of each expression (or really any function defined on the
        set of $H_F$-conjugacy classes of $\theta$-split tori in $G$).
\end{lemma}

\begin{proof}
    Pin down a total indexing set
    \[
        \mathscr I:=\{(M,T,S):M\in\mathcal L^-,\ T\in\mathcal T_M,
        \ S\in\mathcal T_G(T)\}
    \]
    and consider a projection
    \[
        \pi:\mathscr I\longrightarrow\mathcal T_{G/H}\qquad
        \pi(M,T,S)=\text{the unique element of }\mathcal T_{G/H}
        \text{ that is }H_F\text{-conjugate to }S
    \]
    Fix $S_0\in\mathcal T_{G/H}$.  Its $G_F$-conjugacy class
    determines the conjugacy class of $Z_G(A_{S_0})$, and hence
    a unique $M\in\mathcal L^-$.  Conjugating $Z_G(A_{S_0})$ to $M$
    sends $S_0$ to an elliptic maximal torus of $M$; its $M_F$-class has a
    representative $T_0\in\mathcal T_M$.  The
    preceding lemma shows that the number of elements of $\mathcal T_M$ in
    this $G_F$-conjugacy class is
    \[
        |W_G(A_M)|
        \frac{|W_M(T_0)|}{|W_G(T_0)|}
    \]
    Choosing another representative $T\in\mathcal T_M$ of this $G_F$-conjugacy class
    doesn't change this quantity, and so it also equals $|\pi^{-1}(S_0)|$.

    Substitution yields
    \begin{align*}
         & \sum_{(M,T,S)\in\pi^{-1}(S_0)}
        \frac{1}{|W_G(A_M)|\,|W_M(T)|}
        \frac{|W_G(S)|}{|W_H(S)|} = \frac{1}{|W_H(S_0)|}
    \end{align*}
    as required.
\end{proof}

This last decomposition is better suited to our purposes, in part, because we will be able to absorb the sum over
    $\mathcal{T}_G(T)$ into the weight factors in a natural way.

\subsubsection{Weight Factors}

Remember that we have defined weight factors as integrals over symmetric spaces associated to certain tori:
\begin{align*}
    \omega_S(g) & := \int_{Z_H(S) \backslash Z_G(S)} \phi(tg) \; dt
\end{align*}

Let $S[2]$ be the square roots of unity in $S$. For a maximal $\theta$-split torus $S$, squaring identifies
$S[2]_F\backslash S_F$ with $S_F^2$. Also, over a local field $F$, there are finitely many square classes,
so $S_F/S_F^2$ is finite and $S_F^2$ is open. This means that transporting the Haar measure back and forth between
$S[2]_F\backslash S_F$ and $S_F^2$ is not problematic.

Suppose that $S\in\mathcal T_G(T)$ for some $T\in\mathcal T_M$.
Choose $g_S\in G_F$ with $S=g_STg_S^{-1}$.  Since $S$ is a
$\theta$-split maximal torus, $Z_H(S)=S[2]$, and so
$g_S^{-1}Z_H(S)g_S=T[2]$.  Thence
\begin{align*}
    \omega_S(g) & = \int_{T[2] \backslash Z_G(T)} \phi(g_Stg_S^{-1}g) \, dt
\end{align*}
Let $A_M$ be the $F$-split component of the center of $M$. Since $T$ is elliptic in $M$ and $A_M^-=A_M$, the group
    $Z_G(T)_F=T_F$ contains the cocompact subgroup $(A_M^- \cdot T[2])_F$, which is a more convenient foundation
    for our truncation procedure:
\begin{align*}
     & \int_{A_M^- \cdot T[2] \backslash Z_G(T)} \int_{T[2] \backslash A_M^- \cdot T[2]} \phi (g_Sstg_S^{-1}g) \, ds\, dt
\end{align*}
That is, because $A_M^- \cdot T[2] \backslash Z_G(T)$ is compact, the outer integral is not an essential part of
    the truncation procedure, and we may absorb the integral over $(A_M^- \cdot T[2])_F \backslash Z_G(T)_F$ into
    the outer orbital integrals. Changing variables $g \mapsto g_Sg$, we then let
\begin{align*}
    \omega_{M,T}(g) & := \sum_{S \in \mathcal{T}_G(T)} \frac{|W_G(S)|}{|W_H(S)|}
        \int_{T[2]_F\backslash (A_M^- \cdot T[2])_F} \phi(g_Ssg)\; ds
\end{align*}
We will study these weight factors.

\subsection{The Trace Formula}

We now assemble these ingredients into a skeleton for the $\theta$-split side of the trace formula.

\begin{definition}
    For any family of truncation functions
    \[
        \omega_{M,T}\in
        C^\infty\bigl((A_M^-\cdot T[2])\backslash G\bigr)
        \qquad M\in\mathcal L^-\quad T\in\mathcal T_M
    \]
    set $J_-(f,\omega)$ equal to the sum
    \begin{align*}
        \sum_{M \in \mathcal{L}^-} \sum_{T \in \mathcal{T}_M} \frac{1}{|W_G(A_M)| \cdot |W_M(T)|}
            \int_{\mathfrak{t}'} \left| D^G(X) \right|^{\frac{1}{2}} \int_{A_M^- \cdot T[2] \backslash G}
            f(g^{-1}Xg) \, \omega_{M,T} (g) \, dg\, dX
    \end{align*}
    whenever these integrals converge absolutely.
\end{definition}

This compact notation affords a succinct statement of the $\theta$-split side.

\begin{thm} \label{levi_formula}
    Suppose one has a split, reductive algebraic group $\mathbb{G}$ over $\mathcal{O}$, equipped with an involution
        $\theta : \mathbb{G} \rightarrow \mathbb{G}$ over $\mathcal{O}$ and suppose that $\mathbb{G}(F)$ contains a
        $(\theta, F)$-split maximal torus. Let $\mathbb{H} = \mathbb{G}^{\theta}$. Then for every $f \in
        C_c^{\infty}(\mathfrak{g})$ and
    $\phi\in C_c^\infty(H\backslash G)$,
    \begin{align*}
        \int_{H \backslash G} \int_{\mathfrak{h}} \hat{f}(g^{-1}Xg)\, \phi(g)\; dX\, d\dot{g} & = J_- (f, \omega)
    \end{align*}
    where
    \begin{align*}
        \omega_{M,T}(g) & = \sum_{S \in \mathcal{T}_G(T)}
        \frac{|W_G(S)|}{|W_H(S)|}
        \int_{T[2]_F\backslash (A_M^- \cdot T[2])_F}
        \phi(g_Ssg) \; ds
    \end{align*}
\end{thm}

For applications, however, we would like to refine this expression for the $\theta$-split side further,
expressing these weight factors in terms of lattice counts, and
eventually in terms of equivariant line bundles on toric varieties. This is the goal of the next section of this paper.

\section{The Weight Factors} \label{3}

\subsection{More Structure Theory}

Let $A$ be our $(\theta, F)$-split maximal torus whose apartment contains $x_K$.
Then $A\cap K=A(\mathcal O)$, and we may identify $X_*(A)$ with $A / A \cap K$. To save space, write $\nu_a$ for
    the image of $a$ under the map $A \twoheadrightarrow A / (A \cap K)$. When $A_M$ and $A$ are two $F$-split tori
    of $G$ such that $A_M \subset A$, we will treat $X_*(A_M)$ as a subset of $X_*(A)$ because there is a natural
    inclusion $X_*(A_M) \hookrightarrow X_*(A)$. For example, we may say without confusion that elements $a \in
    A_M$ map to elements $\nu_a \in X_*(A)$.

Let's consider the parabolic subgroups adapted to $\theta$ that contain $A$.

\begin{definition}
    Let  $M$ be a $\theta$-stable Levi subgroup of $G$ and write
    \begin{align*}
        \mathcal{P}(M) & := \{P:\text{$P$ is a parabolic with Levi component $M$}\}
    \end{align*}
    We say that $P \in \mathcal{P}(M)$ is $\theta$-split if $P$ and $\theta(P)$ are opposite parabolics with
        respect to $M$. We write
    \begin{align*}
        \mathcal{P}(M)^- & := \{P\in\mathcal P(M):\theta(P)=\bar P\}
    \end{align*}
    where $\bar P$ is the parabolic opposite to $P$.
\end{definition}

As with $\mathcal{P}(M)$, we can describe the elements of  $\mathcal{P}(M)^-$ in terms of Weyl facets.

\begin{prop} \rm \cite{HW93} \it
    Let $A_M$ be the $F$-split component of the $\theta$-stable Levi subgroup $M = Z_G(A_M^-)$. Then parabolics in
        $\mathcal{P}(M)^-$ correspond bijectively with the Weyl chambers of the restricted root system $\Phi(A^-_M,
        G)$. This bijection is given by associating to each Weyl chamber in $X_*(A_M^-)$ of the restricted root
        system the unique Weyl facet of $\Phi(A_M, G)$ in $X_*(A_M)$ containing it, and thence a parabolic subgroup
        in $\mathcal{P}(M)$.
\end{prop}

Let $B$ be a Borel subgroup containing $A$ and let $\mu \in X_*(A)$ be regular in the sense that it does not belong
    to any $X_*(A_M)$ except $X_*(A)$. The Weyl group acts on $X_*(A)$ and the orbit of $\mu$ under this action
    contains a unique element dominant with respect to $B$. Call this element $\mu_B$. For any parabolic $P$
    containing $B$, let $\mu_P$ denote the projection of $\mu_B$ onto the linear span of the Weyl facet associated
    to $P$ in $\mathfrak{a} := X_*(A) \otimes_{\mathbb{Z}} \mathbb{R}$. Note that each $\mu_P$ is in fact
    well-defined (see e.g., section 12 of \cite{Kot06}). By varying $B$ over the Borel subgroups containing $A$, we
    obtain a family of cocharacters $\mu_B$ indexed by $\mathcal{P}(A)$.

Recall the homomorphism
\begin{align*}
    H_G & : G \rightarrow \Lambda_G := \{ \text{cocharacters} \} / \{ \text{coroots} \}
\end{align*}
defined to be trivial on $K$ and equal to the projection
\begin{align*}
    A \twoheadrightarrow A / (A \cap K) \twoheadrightarrow \Lambda_G
\end{align*}
on $A$. By the Cartan decomposition, these two properties determine $H_G$. For any parabolic $P$ with Levi
    complement $M$ and unipotent radical $N$, we define
\begin{align*}
    H_P(mnk) & := H_M(m)
\end{align*}
where $m \in M$, $n \in N$, and $k \in K$.  By the Iwasawa decomposition, this uniquely determines a function
\begin{align*}
    H_P & : G \rightarrow \Lambda_M := \{ \text{cocharacters} \} / \{ \text{coroots of $M$} \}
\end{align*}

For points $x,y$ in the apartment write $x\leq_By$ when
$y-x$ lies in the nonnegative real simple-coroot cone for $B$.
We will drop the subscript when $B$ corresponds
to our chosen positive roots.

For any set $S$ of points in $\Lambda_M$, we define $\text{Hull\,} S$ to be the convex hull of the projections of
    these points to $\mathfrak{a}_M := \Lambda_M \otimes_{\mathbb{Z}} \mathbb{R}$ under the map
\begin{align*}
    \Lambda_M \rightarrow  \Lambda_M \otimes_{\mathbb{Z}} \mathbb{R}= \mathfrak{a}_M
\end{align*}
We will also write $\text{Hull\,}S $ for the preimage of this convex hull under this map. The meaning of
    $\text{Hull\,}S$ will therefore depend on whether it is a subset of $\Lambda_M$ or $\mathfrak{a}_M$.

It will often be convenient to break up $\text{Hull\,}S$ according to the images of its elements in $\Lambda_G$. If
    each element of $S$ maps to the same element in $\Lambda_G$ under the natural map $\Lambda_M \rightarrow
    \Lambda_G$, then write $\text{Hull\,} S^*$ for the set of elements in $\text{Hull\,} S$ that have the same
    image in $\Lambda_G$ as every element of $S$. We will use a set of this form to define the truncating function
    $\phi$.

\subsection{A First Weight Factor}

Fix a Borel subgroup $B\supset A$.
This provides a choice of simple roots $\Delta$, dominant coweights in
$X_*(A)$, and positive coroots. Write $X_*(A)_{dom}$ for the dominant
coweights.

Thanks to the Cartan decomposition, there is a proper map
\begin{align*}
    \text{Cartan} & : G \twoheadrightarrow K \backslash G / K \cong X_*(A)_{dom}
\end{align*}
to the set of dominant coweights of $A$, as well as the ``invariant" map
\begin{align*}
    \text{inv} & : G \times G \rightarrow X_*(A)_{dom} \subset \Lambda_A \\
               & \hspace{2ex} (g, h) \mapsto \text{Cartan\,}(h^{-1}g)
\end{align*}

\begin{definition}
    Set
    \begin{align*}
        \bar{\omega}(g, \mu) & := \begin{cases}
                                      1 & \text{inv\,} (g, \theta(g)) = \text{Cartan\,}(\theta(g)^{-1}g) \in
                                          \text{Hull} \left\{ W_G(A) \cdot \mu \right\}^* \\
                                      0 & \text{otherwise}
                                  \end{cases}
    \end{align*}
\end{definition}
When the algebraic group $\mathbb{G}$ is not split over $F$, the Cartan decomposition is
slightly more complicated, but nevertheless may be used to define an analogous function, as in
\cite{Art91}. This more complicated function still yields a manageable truncation procedure.

Notice that each element in the Weyl orbit of a dominant coweight $\mu$ maps to the same element
under the map $X_*(A) \rightarrow \Lambda_G$, so that we can indeed extract a ``starred" convex
hull of this orbit according to this common image. To show that this function satisfies the
conditions of theorem \ref{formal_identity}, we need to check that it is locally constant and
compactly supported.

\begin{lemma}
    The function $g \mapsto \bar{\omega}(g, \mu)$ belongs to $C_c^{\infty}(H \backslash G)$ for
    any $\mu \in X_*(A)_{dom}$.
\end{lemma}

\begin{proof}
    The proper map Cartan is locally constant, so it's enough to show that precomposition by the
    map
    \begin{align*}
        \tau : H \backslash G \rightarrow G\qquad Hg \mapsto \theta(g)^{-1} g
    \end{align*}
    pulls back $C_c^{\infty}(G)$ to $C_c^{\infty}(H \backslash G)$. But this map is a closed
    immersion (see \cite{Ric82}), completing the proof.
\end{proof}

Our next goal will be to understand the asymptotic behavior of $\omega_{M,T}$ as $\mu$ becomes
very large in the direction of the Borel subgroup $B$.

\subsection{Orthogonal Sets and Arthur's Key Geometric Lemma}

In this section, we state Arthur's key geometric lemma. Detailed accounts include section 5 of
the original article \cite{Art91}, and the sections leading up to and including section 22 in
the expository article \cite{Kot06}.

We will need the notion of a $(G, A)$-orthogonal set. Recall that a $(G, A)$-orthogonal set is a
set of points $x_B$ in $X_*(A)$, indexed by the Borel subgroups $B$ containing $A$, such that
for each pair of adjacent Borel subgroups $B$ and $B'$,
\[
    x_B - x_{B'} = r \check{\alpha}
\]
where $r$ is an integer and $\check{\alpha}$ is the unique coroot that is positive for $B$ and
negative for $B'$. These sets are called positive if $r$ is positive for each adjacent pair.
This notion can also be extended to include a set of points $x_B \in \mathfrak{a}$ by requiring
only that $r$ be a real number, and not necessarily an integer. The properties of these sets are
discussed in depth in \cite{Art91} or \cite{Kot06}.

We will also occasionally impose an additional regularity condition on these sets:

\begin{definition}
    A $(G, A)$-orthogonal set $(x_B)$ is called \it special \rm if $x_B$ is strictly
    $B$-dominant for every Borel subgroup $B \supset A$.
\end{definition}

There is a generalization of a $(G, A)$-orthogonal set that is important in the theory of the
local trace formula. Set $R_U$ and $R_{U'}$ for the roots in the Lie algebras of the unipotent
radicals $U$ and $U'$ of the parabolic subgroups $P$ and $P'$. For a given Levi subgroup $M$,
define a $(G, M)$-orthogonal set to be a family of points $x_P$ in $\Lambda_M$, indexed by
$P\in\mathcal P(M)$, subject to the condition that for each pair of adjacent parabolic subgroups
$P$ and $P'$,
\[
    x_P - x_{P'} = r \beta_{P, P'}
\]
where $r$ is an integer and $\beta_{P, P'}$ is the unique element among the images in
$\Lambda_M$ of $\alpha^\vee$, for $\alpha\in R_U\cap R_{\bar U'}$, such that all the other
images are its positive integral multiples (remember that $\bar U'$ is the unipotent radical
opposite to $U'$).

These generalized orthogonal sets are called positive if $r$ is positive for each adjacent pair
of parabolic subgroups in $\mathcal P(M)$. They are called special if the condition
\[
    \left< \alpha, x_P \right> > 0 \text{ for all } \alpha \in R_U
\]
holds for every $P\in\mathcal P(M)$. As with $(G, A)$-orthogonal sets, this notion can be
extended to include sets of points in $\mathfrak{a}_M$ by requiring that $r$ be a real number
instead of an integer.

Arthur's key geometric lemma concerns coweights that are sufficiently regular in the sense that
$\left< \alpha, \mu \right>$ is very large for each positive root $\alpha$. Exactly how large
$\left< \alpha, \mu \right>$ needs to be will depend on an element $g$ of $G$, and the location
of the vertex in the (enlarged) Bruhat-Tits building $\mathcal{B}(G)$ of $G$ that $g$ represents. This
dependence can be captured by a single function on $\mathcal{B}(G)$.

More precisely, on $\mathfrak{a}$ one can choose a Weyl group invariant Euclidean norm
$\| \cdot \|_E$ which extends uniquely to a $G$-invariant metric on $\mathcal{B}(G)$. We denote
this metric by $d(x_1, x_2)$ where $x_1$ and $x_2$ are points in $\mathcal{B}(G)$. Let
$x_0:=x_K$ be the basepoint of $\mathcal{B}(G)$ (whose stabilizer contains $K$), and set
\begin{align*}
    d(x) & := d(x, x_0)          \\
    d(g) & := d(g\cdot x_0, x_0)
\end{align*}
It is the quantity $d(g)$ that we need to state Arthur's key geometric lemma.

We write $\kappa_G:X_*(A)\to\Lambda_G$ for the natural projection and
$\mathcal C_\mu=\operatorname{Hull}(W\mu)\cap(\mu+Q_G^\vee)$ for translations of the ``starred"
hull.

\begin{prop}[Starred geometric lemma]\label{starred_geometric}
    Let $x_1,x_2\in G_F$ and let $\mu$ be dominant. For some constant $c$,
    \[
        \langle\alpha,\mu\rangle\geq c(1+d(x_1)+d(x_2))
        \quad(\alpha\in\Delta)
    \]
    implies that the points $x_B=\mu_B-H_B(x_2)+H_{\bar B}(x_1)$ compose a special orthogonal
    set. Let $\mathscr P_{\mu;x_1,x_2}$ be its (real) hull. If $a\in A_F$, and $\lambda=\nu_a$,
    then
    \begin{equation}\label{starred_geometric_indicator}
        \operatorname{Cartan}(x_1^{-1}ax_2) \in \mathcal C_\mu \iff
        \lambda \in \mathscr P_{\mu;x_1,x_2} \text{ and }
        \kappa_G(\lambda)=\kappa_G(\mu)+H_G(x_1)-H_G(x_2)
    \end{equation}
    This is uniform in $a$.
\end{prop}
\begin{proof}
    All of its vertices map to $\kappa_G(\mu)-H_G(x_2)+H_G(x_1)$. Since
    $H_G(x_1^{-1}ax_2)=-H_G(x_1)+\kappa_G(\lambda)+H_G(x_2)$, the starred hull in that lemma is
    exactly the product of indicators of \eqref{starred_geometric_indicator}. Then apply
    Arthur's key geometric lemma \cite{Kot06}.
\end{proof}

This geometric lemma relates the Cartan decomposition to the combinatorics of special orthogonal
sets, and will help us derive an asymptotic expression for the preliminary weight factors.

\subsection{\texorpdfstring{First Steps Toward Lattice Counts}{First Steps Toward Lattice Counts}}

When $\phi(g)=\bar{\omega}(g,\mu)$, the trace formula contains the weight
\[
    \omega_{M,T}(g,\mu):=\sum_{S\in\mathcal T_G(T)}
    \frac{|W_G(S)|}{|W_H(S)|}
    \int_{Q_{M,T}}\bar\omega(g_Ssg,\mu)\,ds
\]
where
\[
    Q_{M,T}:=T[2]_F\backslash(A_M^-\cdot T[2])_F
\]
We partition the domain of the integral over $Q_{M,T}$ according to
the valuation of $s^2$. Set
\[
    \begin{aligned}
        \Lambda_{M,T}    & :=\{\nu_{s^2}:s\in Q_{M,T}\}\subseteq X_*(A_M)                      \\
        Q_{M,T}(\lambda) & :=\{s\in Q_{M,T}:\nu_{s^2}=\lambda\}\qquad(\lambda\in\Lambda_{M,T})
    \end{aligned}
\]
Squaring is well defined on \(Q_{M,T}\), with values in \(A_M(F)\), so these sets partition the
domain:
\[
    Q_{M,T}=\coprod_{\lambda\in\Lambda_{M,T}}Q_{M,T}(\lambda)
\]
Each piece is compact and open, and is a translate of \(Q_{M,T}(0)\). So
\begin{align}\label{first_lattice_count_weight_formula}
    \omega_{M,T}(g,\mu)=\sum_{S\in\mathcal T_G(T)}
    \frac{|W_G(S)|}{|W_H(S)|}
    \sum_{\lambda\in\Lambda_{M,T}}
    \int_{Q_{M,T}(\lambda)}
    \bar\omega(g_Ssg,\mu)\,ds
\end{align}
When $\mu$ is sufficiently regular, these integrals will equal either zero or a common volume,
and the weight factor will be a function of the number of lattice points in $\Lambda_{M,T}$ that
correspond to nonzero terms.

\subsection{Further Refinements of this Weight Factor}

In this section, we describe the asymptotic weight factor $\omega_{M,T}^{asymp}$ in terms of the
involution $\theta$. Specifically, we show that it depends only on those ingredients of the
structure theory of $G$ that are associated to the $(\theta,F)$-split torus $A_M^-$. We begin by
discussing $(G,A)$-orthogonal sets in more detail.

Remember that we've written $\kappa_G:X_*(A)\to\Lambda_G$ for the natural projection.

\begin{lemma}
    \cite{Kot06} Let $\{x_B:B\in\mathcal P(A)\}$ be a positive integral $(G,A)$-orthogonal set,
    and let $\xi\in\Lambda_G$ be the common image of the $x_B$. Then
    \[
        \operatorname{Hull}\{x_B:B\in\mathcal P(A)\}^*
        =\{\lambda\in X_*(A):\kappa_G(\lambda)=\xi
        \text{ and }\lambda\leq_{B}x_B\text{ for every }B\in\mathcal P(A)\}
    \]
    where the inequalities take place in $\mathfrak a=X_*(A)\otimes_{\mathbb Z}\mathbb R$.
\end{lemma}

$\theta$ defines an involution on $X_*(A)$ by composition and thence an involution on
$\mathfrak{a}$. Let $\mathfrak{a}^-$ be the $(-1)$-eigenspace of this restricted involution. For
$B\supset A$ and $y\in\mathfrak a$, set
\[
    \mathcal P_B^-(y) :=\{x\in\mathfrak a^-:x\leq_{B}y\}
\]
We can express the convex hull of a $(G,A)$-orthogonal set in terms of these cones.

\begin{lemma}
    Let $\{x_B:B\in\mathcal P(A)\}$ be a positive $(G,A)$-orthogonal set with common image
    $\xi\in\Lambda_G$, and put
    \[
        \Lambda_A^-:=X_*(A)\cap\mathfrak a^-
        \qquad\text{and}\qquad
        \Lambda_A^-(\xi):=\{\lambda\in\Lambda_A^-:\kappa_G(\lambda)=\xi\}
    \]
    Then
    \[
        \operatorname{Hull}\{x_B:B\in\mathcal P(A)\}^*\cap\Lambda_A^-
        =\Lambda_A^-(\xi)\cap
        \bigcap_{B\in\mathcal P(A)}\mathcal P_B^-(x_B)
    \]
\end{lemma}

\begin{proof}
    The coroot inequalities follow from the definitions of each set $\mathcal P_B^-(x_B)$, and
    $\kappa_G(\lambda)=\xi$. Intersecting with $\Lambda_A^-$ completes the proof.
\end{proof}

Let $P$ be a parabolic subgroup with Levi component $M$. The map $H_P : G \rightarrow \Lambda_M$
may be composed with the natural map $\Lambda_M \rightarrow \mathfrak{a}_M$, providing a second
map $H_P : G \rightarrow \mathfrak{a}_M \subset \mathfrak{a}$. We will not distinguish between
these two maps with additional notation.

\begin{lemma}[Involution covariance]\label{refinement_two}
    For any $P\in\mathcal P(M)$,
    \[
        H_{\bar P}(\theta(g))=-H_P(g)
    \]
    In particular, $H_{\bar P}(\theta(g))-H_P(g)=-2H_P(g)$.
\end{lemma}
\begin{proof}
    $\theta$ takes $P$ to $\bar P$, preserves $M$ and $K$, and acts by $-1$ on $\Lambda_M$. So
    applying $\theta$ to $g=mnk$ gives $H_{\bar P}(\theta(g))=H_M(\theta(m))=-H_M(m)$.
\end{proof}

We can now evaluate the inner integral in \eqref{first_lattice_count_weight_formula}. For any
$s \in T$, $\theta(g_Ssg_S^{-1})=g_Ss^{-1}g_S^{-1}$. So
\[
    \tau(g_Ssg)=\theta(g_Sg)^{-1}g_Ss^2g
\]
The orthogonal set associated to $x_1=(g_S)^{-1}\theta(g_Sg)$ and $x_2=g$ in Arthur's lemma is
then
\[
    x_{B,S}(g,\mu)
    :=\mu_B-H_B(g)+H_{\bar B}((\theta(g_S)^{-1}g_S)^{-1}\theta(g))
\]
The common image of these points under $\kappa_G$ is
\[
    \xi_{M,T,S}(g,\mu)
    :=\kappa_G(\mu)-H_G(g)+H_G((\theta(g_S)^{-1}g_S)^{-1}\theta(g))
\]
Write $\iota_{M,T}:\Lambda_{M,T}\hookrightarrow X_*(A)$ for the natural inclusion, whether or
not it's been extended to real vector spaces. Using this inclusion, we define the real polytope
\[
    \mathscr P^-_{M,T,S}(g,\mu)
    :=\iota_{M,T}^{-1}\left(
    \bigcap_{B\in\mathcal P(A)}
    \mathcal P_B^-(x_{B,S}(g,\mu))\right)
\]
and its ``starred" variant
\[
    P^-_{M,T,S}(g,\mu)
    :=\left\{\lambda\in
    \mathscr P^-_{M,T,S}(g,\mu)\cap\Lambda_{M,T}:
    \kappa_G(\iota_{M,T}(\lambda))
    =\xi_{M,T,S}(g,\mu)\right\}
\]
The asymptotic weight factor then has this form:
\[
    \omega_{M,T}^{\mathrm{asymp}}(g,\mu)=
    \sum_{S\in\mathcal T_G(T)}
    \frac{|W_G(S)|}{|W_H(S)|}
    \sum_{\lambda\in P^-_{M,T,S}(g,\mu)}
    \operatorname{vol}_{ds}Q_{M,T}(\lambda)
\]
To conclude this subsection, we show that the ingredients of this expression are computationally
tractable, first for $P^-_{M,T,S}(g,\mu)$, second for $\operatorname{vol}_{ds}Q_{M,T}(\lambda)$.

\begin{lemma}\label{special_projection}
    Let $\operatorname{pr}_M:\mathfrak a\to\mathfrak a_M$ be orthogonal projection and set
    $x_{P,S}:=\operatorname{pr}_M(x_{B,S})$ for $B\subset P$. These points form a
    $(G,M)$-orthogonal set. Further, if $(x_{B,S})$ is special, then
    \[
        \mathscr P^-_{M,T,S}(g,\mu)
        =\operatorname{Hull}\{x_{P,S}(g,\mu):P\in\mathcal P(M)\}
    \]
\end{lemma}
\begin{proof}
    See also \cite[Proposition~12.1]{Kot06}. The projection of the hull of the $x_{B,S}$ is the
    hull of the $x_{P,S}$. For fixed $P$, the points $x_{B,S}-x_{P,S}$ with $B\subset P$
    comprise a special orthogonal set for the derived root system of $M$, and their hull
    contains zero. Thus every $x_{P,S}$ lies in the hull of the $x_{B,S}$. And since the convex
    hull of the $x_{P,S}$ is also the projection of this hull onto $\mathfrak a_M$,
    $\operatorname{Hull}\{x_{B,S}\}\cap\mathfrak a_M=\operatorname{Hull}\{x_{P,S}\}.$
\end{proof}

\begin{lemma}\label{measured_valuation_fibers}
    Let $v_{M,T}:Q_{M,T}\to\Lambda_{M,T}$ be the map $s\mapsto\nu_{s^2}$. For each
    $S\in\mathcal T_G(T)$, all the fibers of this map have the same volume:
    \[
        C_{M,T,S}
        :=\operatorname{vol}_{ds}Q_{M,T}(0)
        =\operatorname{vol}_{ds}Q_{M,T}(\lambda)
        \qquad(\lambda\in\Lambda_{M,T})
    \]
\end{lemma}

\begin{proof}
    The image of $A_M(F)$ in $Q_{M,T}$ has finite index, and $v_{M,T}$ restricts to twice the
    usual valuation on this subgroup. So its kernel is compact and open. Since every fiber is a
    translation of this kernel, the claim follows.
\end{proof}

\subsection{Polynomial-like Weight Factors}
\label{polynomial_wf}

The simplest way to describe the functions $\omega_{M,T}^{\mathrm{asymp}}$ in terms of
polynomials for $\mu$ sufficiently regular is perhaps an appeal to the theory of toric
varieties. One can extract from \cite{Ful93} the relevant results, which are described in this
context in \cite{Kot06}. They let us express the number of lattice points inside an 
orthogonal set in terms of the Euler-Poincar\'{e} characteristic
\[
    EP(\mathcal{L}) := \sum_i (-1)^i \dim H^i(V, \mathcal{L})
\]
of a line bundle $\mathcal{L}$ on a (toric) variety $V$. For any automorphism $s$ of $V$ with a
chosen lift to $\mathcal L$, we will also need the function
\[
    EP(s, \mathcal{L}) :=  \sum_i (-1)^i \text{tr\,}(s; H^i(V, \mathcal{L}))
\]
so that $EP(\mathcal{L}) = EP(1, \mathcal{L})$, where $1$ is the identity automorphism.

For each Levi subgroup $M$ in $\mathcal L^-$, let $\hat G$ and $\hat M$ denote the Langlands
duals of $G$ and $M$, and let
\[
    D_{\hat M}:=Z(\hat M)/Z(\hat G)
\]
This is a torus: its character lattice $K_M=Q_G^\vee/Q_M^\vee$ is free, since the simple coroots
of a Levi are a subset of those of $G$.
\[
    X^*(D_{\hat M})=
    K_M:=\ker(\Lambda_M\to\Lambda_G)
\]
Next, we build a lattice that situates our weight factors within this theory.

\subsubsection{The Weight Factor Polytope}
Remember our notation for the natural inclusion
$\iota_{M,T}:\Lambda_{M,T}\hookrightarrow X_*(A)$ and also write
$\kappa_{M,T}:=\kappa_G\circ\iota_{M,T}$. The kernel of this map is $\Lambda_{M,T} \cap Q_G^\vee$, and every 
nonempty fiber is a translation of this kernel. This means we can relate every translated ``starred" 
convex hull to $\ker(\kappa_{M,T})$. In fact, choosing some \(a_{M,T,S}(g,\mu) \in \Lambda_{M,T}\) 
in the fiber over \(\xi_{M,T,S}(g,\mu)\) gives a family of bijections
\[
    P^-_{M,T,S}(g,\mu)
    \xrightarrow{\;\lambda\mapsto\lambda-a_{M,T,S}(g,\mu)\;}
    \bigl(
    \mathscr P^-_{M,T,S}(g,\mu)-a_{M,T,S}(g,\mu)\bigr)
    \cap\ker(\kappa_{M,T})
 \]
Because our weight factors count lattice points in \(\mathfrak a_M\) and projecting to \(\mathfrak a_M\) 
can introduce denominators into the vertices of our polytope, we need to make two adjustments: 
first, we tensor \(\ker(\kappa_{M,T})\) by \(\mathbb Q\) to allow for these denominators;
second, we restrict the translated polytope to the real span of the kernel, a subspace of \(\mathfrak a_M\). 
To this end, set
\[
    \Gamma_{M,T}:=
    \bigl(\ker(\kappa_{M,T})\otimes_{\mathbb Z}\mathbb Q\bigr)
    \cap\Lambda_{M,T}
\]
An element $\lambda\in\Lambda_{M,T}$ belongs to $\Gamma_{M,T}$ precisely
when some positive multiple of $\lambda$ belongs to $\ker(\kappa_{M,T})$, or, equivalently, when
$\kappa_{M,T}(\lambda)$ is torsion. In other words,
\[
    \Gamma_{M,T}/\ker(\kappa_{M,T})
    \cong\operatorname{tors}(\operatorname{im}\kappa_{M,T})
\]
and this kernel has finite index in $\Gamma_{M,T}$. 
In this way, we have included rational, possibly non-integral linear combinations of elements of the kernel. 
To situate the translated polytope in the real span of \(\Gamma_{M,T}\), we intersect with that subspace:
\[
    \mathscr Q_{M,T,S}(g,\mu):=
    \bigl(\mathscr P^-_{M,T,S}(g,\mu)-a_{M,T,S}(g,\mu)\bigr)\cap
    \bigl(\Gamma_{M,T}\otimes_{\mathbb Z}\mathbb R\bigr)
\]

We have some freedom to choose these $a_{M, T, S}$, and it will be helpful to endow them
with additional structure.
Choose a finite index sublattice $\Lambda^0\subset X_*(A)$ such that
$\kappa_G(\Lambda^0)\subseteq\operatorname{im}\kappa_{M,T}$ for all $M$ and $T$. Since
$\Lambda^0$ is free, it is projective, and the restriction $\kappa_G|_{\Lambda^0}$ lifts through
$\kappa_{M,T}$ to a homomorphism into $\Lambda_{M,T}$. Write
\(\ell_{M,T}:\Lambda^0\to\Lambda_{M,T}\) for this lift. For fixed \(g\), choose a representative
\(\mu_0\) of each coset of \(\Lambda^0\) with nonempty fiber, and then set
\[
    a_{M,T,S}(g,\mu_0+\eta):=a_{M,T,S}(g,\mu_0)+\ell_{M,T}(\eta)\qquad \eta\in\Lambda^0
\]
These modified translations are affine, with the same linear part \(\ell_{M,T}\).

\subsubsection{The Toric Variety}
With these $\mathscr Q_{M,T,S}$ in hand, we turn to the task of counting the lattice points they
contain using toric geometry. The vertices of these polytopes might not belong to $\Lambda_{M, T}$, 
so we work with a slightly enlarged set:
\[
    \Gamma'_{M,T}:=\Gamma_{M,T}
    +\operatorname{pr}_M(Q_G^\vee)
    \subset\Gamma_{M,T}\otimes\mathbb Q
\]
Then for any $B\subset P$, the points $x_{B,S}-a_{M,T,S}$ are in $Q_G^\vee$, and so
$x_{P,S}-a_{M,T,S}\in\Gamma'_{M,T}$. On the real vector space
$\operatorname{Hom}(\Gamma'_{M,T},\mathbb R)$, consider the fan cut out by the
hyperplanes orthogonal to the projected coroots of \(G\); its maximal cones correspond to 
parabolics in \(\mathcal P(M)\). Choose a complete nonsingular subdivision and let \(Y_{M,T,S}\)
be the corresponding toric variety.

\subsubsection{Weight Factors from Toric Geometry}
For each $P\in\mathcal P(M)$, the points $x_{P,S}-a_{M,T,S}$ belong to
$\Gamma'_{M,T}$ and define an equivariant line bundle on
$Y_{M,T,S}$. For sufficiently regular $\mu$, this line bundle is globally
generated, and Lemma~\ref{special_projection} identifies its polytope with
$\mathscr Q_{M,T,S}(g,\mu)$. Its Euler-Poincar\'e characteristic then
counts the lattice points in $\mathscr Q_{M,T,S}(g,\mu)$, in a sense made
precise by the following proposition.

\begin{prop}
    Let $Y_\Sigma$ be a complete nonsingular toric variety with character lattice $\Lambda$, and
    let $\mathcal L$ be a globally generated equivariant line bundle with associated lattice
    polytope $P(\mathcal L)$. Then
    \[
        EP(\mathcal L)=|P(\mathcal L)\cap\Lambda|
    \]
    If $L\subseteq\Lambda$ has finite index, $a\in\Lambda$, and
    $\mathcal Z_L:=\operatorname{Hom}(\Lambda/L,\mathbb C^\times)$, then
    \[
        \frac{1}{|\mathcal Z_L|}
        \sum_{z\in\mathcal Z_L}\langle z,a\rangle^{-1}EP(z,\mathcal L)
        =|P(\mathcal L)\cap(a+L)|
    \]
\end{prop}

\begin{proof}
    Since \(\mathcal L\) is globally generated, its higher cohomology vanishes. The torus
    therefore acts on \(EP(\mathcal L) = H^0(Y,\mathcal L)\) with one-dimensional weight spaces
    corresponding to the lattice points of \(P(\mathcal L)\). This means that
    \[
        EP(z,\mathcal L)=\sum_{\gamma \in P(\mathcal L) \cap \Lambda} \langle z,\gamma\rangle
    \]
    In particular, \(EP(\mathcal L)=|P(\mathcal L)\cap\Lambda|\). The second identity is a
    consequence of character orthogonality.
\end{proof}

We can now apply this proposition to count lattice points in the ``starred" hull. Let
$\mathbb T_{M,T,S}$ be the (algebraic) torus whose character lattice is $\Gamma'_{M,T}$,
and set
\[
    \mathcal Z_{M,T,S}:=\operatorname{Hom}\big( \Gamma'_{M,T}/\ker(\kappa_{M,T}), \mathbb C^\times\bigr)
\]
Whenever
$\xi_{M,T,S}(g,\mu)\in\operatorname{im}\kappa_{M,T}$,
write $\mathcal L_{M,T,S}(g,\mu)$ for the
$\mathbb T_{M,T,S}$-equivariant line bundle with Cartier data
\[
    x_{P,S}(g,\mu)-a_{M,T,S}(g,\mu)
    \qquad(P\in\mathcal P(M))
\]
We adopt the convention that every expression
$EP(z,\mathcal L_{M,T,S}(g,\mu))$ is understood to be zero if
$\xi_{M,T,S}(g,\mu)\notin\operatorname{im}\kappa_{M,T}$.

\begin{coro}
    Suppose $S\in\mathcal T_G(T)$. If $\mu$ is sufficiently regular and
    $\xi_{M,T,S}(g,\mu)\in\operatorname{im}\kappa_{M,T}$ then
    \begin{align*}
        \int_{Q_{M,T}}\bar\omega(g_Ssg,\mu)\,ds
         & = C_{M,T,S}|P^-_{M,T,S}(g,\mu)|                \\
         & = C_{M,T,S}
        |\mathscr Q_{M,T,S}(g,\mu)\cap\ker(\kappa_{M,T})| \\
         & = \frac{C_{M,T,S}}{|\mathcal Z_{M,T,S}|}
        \sum_{z\in\mathcal Z_{M,T,S}}
        EP\bigl(z,\mathcal L_{M,T,S}(g,\mu)\bigr)
    \end{align*}
    If $\xi_{M,T,S}(g,\mu)\notin\operatorname{im}\kappa_{M,T}$, the integral vanishes.
\end{coro}

\begin{proof}
    First, given sufficiently regular $\mu$, Arthur's geometric lemma expresses the
    integral as a sum of $|P^-_{M,T,S}(g,\mu)|$ terms, each equal to $C_{M,T,S}$. Second,
    translation by $a_{M,T,S}(g,\mu)$ identifies $P^-_{M,T,S}(g,\mu)$ with
    $\mathscr Q_{M,T,S}(g,\mu)\cap\ker(\kappa_{M,T})$. Third, $\mathscr Q_{M,T,S}(g,\mu)$ is
    the polytope corresponding to $\mathcal L_{M,T,S}(g,\mu)$ and $\ker(\kappa_{M,T})$ has
    finite index in $\Gamma'_{M,T}$, so we can apply a finite Fourier transform.
    Last, notice that $\bar\omega(g_Ssg,\mu)\neq0$ implies $\kappa_{M,T}(\nu_{s^2})=\xi_{M,T,S}(g,\mu)$,
    so the integral vanishes whenever $\xi_{M,T,S}(g,\mu)\notin\operatorname{im}\kappa_{M,T}$.
\end{proof}

One advantage of using toric geometry to compute these lattice counts is that
Euler--Poincar\'{e} characteristics have a particularly simple dependence
on their underlying line bundles. Let $E_{M,T,S}$ be the group of
isomorphism classes of $\mathbb T_{M,T,S}$-equivariant line bundles on $Y_{M,T,S}$. The
Riemann--Roch formula implies that $EP(\mathcal L)$ can be interpreted as a degree
$\dim Y_{M,T,S}$ polynomial on
\[
    (E_{M,T,S}/\Gamma'_{M,T})\otimes_{\mathbb Z}\mathbb Q
\]

For a finite-order element $z\in\mathcal Z_{M,T,S}\subset\mathbb T_{M,T,S}$, the equivariant
Riemann--Roch formula gives a similar description for $EP(z,\mathcal L)$, except that we now
have to introduce finite-order characters of $E_{M,T,S}$ as coefficients for these polynomials.
This means that $EP(z,\mathcal L)$ is polynomial on each coset of $\Lambda^0$ with degree at most
\[
    \dim Y_{M,T,S}=\dim(A_M/A_G)
\]
Importantly, this bound is independent of $g$. Refining $\Lambda^0$ if necessary, we can assume
that all these equivariant Euler--Poincaré characteristics are polynomials in $\mu$ on the cosets
of $\Lambda^0$, in a way that is also independent of $g$.

Here are the weight factors that will replace our original truncation weights:

\begin{definition}
    Set
    \[
        \nu_{M,T}(g,\mu):= \sum_{S\in\mathcal T_G(T)}\frac{|W_G(S)|}{|W_H(S)|}
        \frac{C_{M,T,S}}{|\mathcal Z_{M,T,S}|}
        \sum_{z\in\mathcal Z_{M,T,S}}
        EP\bigl(z,\mathcal L_{M,T,S}(g,\mu)\bigr)
    \]
\end{definition}

For each fixed $g$, and for $\mu$ sufficiently regular relative to $g$, these weight factors
equal $\omega_{M,T}^{\mathrm{asymp}}(g,\mu)$ as well as the original weights
$\omega_{M,T}(g,\mu)$. In particular, along a ray $\mu=\mu_1+d\mu_2$, where $\mu_2\in\Lambda^0$
is strictly dominant, $\nu_{M,T}(g,\mu)$ is a polynomial in $d$. This is the property we use to
pass from our truncated identity to the trace formula.

\subsubsection{Two Additional Observations}
First, we should point out that these new weights are locally constant in $g$ and
left $(A_M^-\cdot T[2])_F$-invariant. Both properties are inherited from the original weights by
their quasipolynomial continuations.

Second, we can also describe the quasipolynomial dependence of these new weight factors with finite
character theory. If $F(\mu)$ is a polynomial on each coset of $\Lambda^0$ in $\Lambda=X_*(A)$,
then finite Fourier inversion gives
\[
    F(\mu)=\sum_{\chi\in\widehat{\Lambda/\Lambda^0}}
    \chi(\mu)P_\chi(\mu)
    \qquad
    P_\chi(\mu)=\frac1{[\Lambda:\Lambda^0]}
    \sum_{a\in\Lambda/\Lambda^0}\chi(-a)Q_a(\mu)
\]
where $Q_a$ is the polynomial agreeing with $F$ on $a+\Lambda^0$. These polynomials are unique,
since a polynomial is determined by its values on any full-rank (affine) sublattice.

\section{Analysis} \label{4}

\subsection{The Adjoint Quotient}

We begin by recalling some facts about the ``adjoint quotient'' map, over which orbital
integrals are especially well-behaved. Set
$\mathbb{A}_G := \text{Spec\,} \mathcal{O}_{\mathfrak{g}}^G$, where $\mathcal{O}_{\mathfrak{g}}$
is the $F$-algebra of polynomial functions on $\mathfrak{g}$ and $\mathcal{O}_{\mathfrak{g}}^G$
is its set of $G$-invariants. The inclusion
\begin{align*}
    \mathcal{O}_{\mathfrak{g}}^G & \hookrightarrow \mathcal{O}_{\mathfrak{g}}
\end{align*}
induces this adjoint quotient map
\begin{align*}
    \pi_G : \mathfrak{g} & \rightarrow \mathbb{A}_G
\end{align*}
which in many cases maps each element of $\mathfrak{g}$ to its characteristic polynomial. On a
Cartan subalgebra, the absolute value of its Jacobian is a constant multiple of
$|D^G(X)|^{\frac{1}{2}}$ \cite[\S14.1]{Kot06}. The fiber $\pi_G^{-1} (x)$ is a union of
conjugacy classes for any $x \in \mathbb{A}_G$. In particular,
\begin{align*}
    \pi_G^{-1} (0) = \mathcal{O}_1 \cup \cdots \cup \mathcal{O}_r =: \mathcal{N}
\end{align*}
where $\mathcal{O}_1, \mathcal{O}_2, \dots, \mathcal{O}_{r-1}$, and $\mathcal{O}_r$ are the
nilpotent $G$-orbits of $\mathfrak{g}$ ordered by increasing dimension. The restriction of
$\pi_G$ to a Cartan subalgebra is \'{e}tale precisely on its regular locus.

Let $M$ be the centralizer of a torus in $G$. The inclusion
$\mathcal{O}_{\mathfrak g}^G\hookrightarrow \mathcal{O}_{\mathfrak m}^M$ induces a finite
morphism
\[
    \rho_M^G:\mathbb A_M\longrightarrow\mathbb A_G
\]
Define
\[
    \mathbb A_M^{G\text{-reg}}
    :=\{a\in\mathbb A_M:\rho_M^G\text{ is \'{e}tale at }a\}
\]
and
\[
    \mathfrak m^{G\text{-reg}}
    :=\pi_M^{-1}(\mathbb A_M^{G\text{-reg}})
    =\{X\in\mathfrak m:D_M^G(X)\neq0\}
    \qquad
    D_M^G(X):=\det\bigl(\operatorname{ad}(X);\mathfrak g/\mathfrak m\bigr)
\]
On a Cartan subalgebra of $M$, the absolute value of the Jacobian of $\rho_M^G$ is a constant
multiple of $|D_M^G(X)|^{1/2}$ \cite[\S14.5]{Kot06}.

\subsection{Shalika Germs}

Because we have ordered the nilpotent $G$-orbits
$\mathcal{O}_1, \mathcal{O}_2, \dots,\mathcal{O}_{r-1}$, and $ \mathcal{O}_r$ of $\mathfrak{g}$
by increasing dimension, the sets $\mathcal{O}_1 \cup \cdots \cup \mathcal{O}_i$ are closed in
$\mathfrak{g}$ for all $i$. In fact, $\mathcal{O}_1 \cup \cdots \cup \mathcal{O}_i$ is closed in
$\mathcal{O}_1 \cup \cdots \cup \mathcal{O}_{i+1}$ and $\mathcal{O}_{i+1}$ is a complementary
open subset. These sets form a stratification of the nilpotent elements of $\mathfrak{g}$. We
review some basic properties of these orbits.

$G$ acts on each $\mathcal{O}_i$ by conjugation. Therefore, the exact sequence
\begin{align*}
    0 \rightarrow \mathcal{D}(\mathcal{O}_1) \rightarrow \mathcal{D}(\mathcal{O}_1 \cup \mathcal{O}_2) \rightarrow
        \mathcal{D}(\mathcal{O}_2) \rightarrow 0
\end{align*}
gives rise to an exact sequence
\begin{align*}
    0 \rightarrow \mathcal{D}(\mathcal{O}_1)^G \rightarrow \mathcal{D}(\mathcal{O}_1 \cup \mathcal{O}_2)^G
        \rightarrow \mathcal{D}(\mathcal{O}_2)^G
\end{align*}
According to \cite{Rao72}, any $G$-invariant distribution on $\mathcal{O}_2$ lifts to a
$G$-invariant distribution on $\mathfrak{g}$. The right-hand map in this sequence is therefore
surjective:
\begin{align*}
    0 \rightarrow \mathcal{D}(\mathcal{O}_1)^G \rightarrow \mathcal{D}(\mathcal{O}_1 \cup \mathcal{O}_2)^G
        \rightarrow \mathcal{D}(\mathcal{O}_2)^G \rightarrow 0
\end{align*}
Because $\dim \mathcal{D}(\mathcal{O}_1)^G = \dim \mathcal{D}(\mathcal{O}_2)^G = 1$, this exact
sequence implies that $\dim \mathcal{D}(\mathcal{O}_1 \cup \mathcal{O}_2)^G = 2$. In this
manner, one proves by induction that
\begin{align*}
    \dim \mathcal{D} (\mathcal{O}_1 \cup \cdots \cup \mathcal{O}_i )^G & = i
\end{align*}
Let $\mu_j$ be the integral over the $j^{th}$ nilpotent orbit. Then the distributions in the set
$\{ \mu_j : j \leq i \}$ span the space
$\mathcal{D}(\mathcal{O}_1 \cup \cdots \cup \mathcal{O}_i)^G$. Choose functions $f_i$ such that
\begin{align*}
    \mu_j (f_i) & = \delta_{ij}
\end{align*}
The orbital integrals over each $f_i$ can be used to define Shalika germs on
$\mathfrak{h}^{\perp}_{\theta-reg}$.
\begin{definition}
    Let $\tilde{\Gamma}_i$ denote the germ of the function
    \begin{align*}
        X \mapsto \begin{cases}
                      \int_{Z_G(X) \backslash G} f_i(g^{-1}Xg) \; \dot{dg} & X \in
                          \mathfrak{h}^{\perp}_{\theta-reg}\text{ semisimple} \\
                      0                                                    & \text{otherwise}
                  \end{cases}
    \end{align*}
    at the origin. We call $\tilde{\Gamma}_i$ the Shalika germ corresponding to the nilpotent
    orbit $\mathcal{O}_i$.
\end{definition}

At semisimple elements, we can write orbital integrals as integrals over compact sets, which are
well-behaved. This observation of Harish-Chandra is not true at nilpotent elements, whose orbits
are not necessarily closed. We will nevertheless follow his development, which starts with the
following critical lemma.

\begin{lemma}\label{semisimple_compactness}
    \cite{HC70} Let $M \subset G$ be the centralizer of a torus in $G$, and let
    $\omega_{\mathfrak g}\subset\mathfrak g$ and $\omega_M\subset\mathbb A_M^{G\text{-reg}}(F)$
    be compact sets. Then
    \begin{align*}
        \{g\in M\backslash G:g^{-1}Xg\in\omega_{\mathfrak g}
        \text{ for some }X\in\pi_M^{-1}(\omega_M)\}
    \end{align*}
    has compact closure in $M \backslash G$.
\end{lemma}

The next few propositions are essentially classical, and follow \cite{Kot06} very closely. But
here we do not impose the condition that $X$ is regular semisimple in $\mathfrak{g}$; rather, we
assume that $X$ is $\theta$-regular and semisimple in $\mathfrak{h}^{\perp}$, not even 
assuming the existence of a $(\theta, F)$-split maximal torus. In this sense, our
development is a useful negative result: the standard classical arguments carry over to our
symmetric-space setting without new difficulties.

\begin{prop}
    Let $f \in C_c^{\infty}(\mathfrak{g})$ and let $T$ be a maximal $\theta$-split torus with
    Lie algebra $\mathfrak{t}$. Let $u$ be any locally constant left $Z_G(T)$-invariant function
    on $G$, which we will eventually assume to be a weight factor. Then
    \begin{align*}
        X & \mapsto O_X(f) := \int_{Z_G(T) \backslash G} f(g^{-1}Xg) \, u(g) \; dg
    \end{align*}
    is locally constant on
    $\mathfrak{t}' := \mathfrak{t} \cap \mathfrak{h}^{\perp}_{\theta-reg}$.
\end{prop}

\begin{proof}
    Fix $X_0\in\mathfrak t'$ and put $M=Z_G(T)$. Since $Z_G(X_0)=Z_G(T)$, one has
    $D_M^G(X_0)\neq0$, even though $X_0$ is central in $\mathfrak m$. Hence
    $\pi_M(X_0)\in\mathbb A_M^{G\text{-reg}}(F)$. Choose a compact open neighborhood $\omega_M$
    of $\pi_M(X_0)$ contained in this relative regular locus, and then a compact open
    neighborhood $V$ of $X_0$ in $\mathfrak t'$ such that $\pi_M(V)\subset\omega_M$.

    Apply Harish-Chandra's lemma with $\omega_{\mathfrak g}=\operatorname{supp}(f)$. The set
    \begin{align*}
        \mathcal X:=\{g\in M\backslash G:g^{-1}Yg\in
        \operatorname{supp}(f)\text{ for some }
        Y\in\pi_M^{-1}(\omega_M)\}
    \end{align*}
    has compact closure. For every $X\in V$, the orbital integral may therefore be written over
    this compact set:
    \begin{align*}
        \int_{M\backslash G}f(g^{-1}Xg)u(g)\,dg
         & =\int_{\overline{\mathcal X}}f(g^{-1}Xg)u(g)\,dg
    \end{align*}
    The functions $f$ and $u$ are locally constant, and compactness makes their local constancy
    uniform on $\overline{\mathcal X}$. The right-hand side is consequently constant for $X$ in
    a sufficiently small neighborhood of $X_0$ in $V$.
\end{proof}

\begin{coro}
    For any maximal $\theta$-split torus $T$ with Lie algebra $\mathfrak{t}$, we can represent
    $\tilde{\Gamma}_i$ by a function that restricts to a locally constant function on
    $\mathfrak{t}'$.
\end{coro}

For the remainder of this subsection, set $u=1$. Thus every occurrence of $O_X(f)$ in the
Shalika-germ expansions below denotes the invariant orbital integral
\[
    O_X(f):=\int_{Z_G(X)\backslash G}f(g^{-1}Xg)\,\dot{dg}
\]
The preceding proposition applies more generally to weighted orbital integrals, but the germ
expansions used here concern invariant orbital distributions.

Near the origin of $\mathfrak{g}$, orbital integrals may be expressed as sums of nilpotent
orbital integrals. This is the Shalika germ expansion.

\begin{prop}
    Let $T$ be a maximal $\theta$-split torus with Lie algebra $\mathfrak{t}$. For each
    $f \in C_c^{\infty}(\mathfrak{g})$, there is a closed and open neighborhood $U$ of
    $0 \in \mathfrak{g}$ such that
    \begin{align*}
        O_X(f) & = \sum_{i=1}^r \mu_i(f) \cdot \tilde{\Gamma}_i(X)
    \end{align*}
    for all $X \in U \cap \mathfrak{t}'$.
\end{prop}

\begin{proof}
    For any $f \in C_c^{\infty}(\mathfrak{g})$, the function
    \begin{align*}
        \phi & = f - \sum \mu_i(f) \cdot f_i
    \end{align*}
    has vanishing nilpotent orbital integrals, and so every element in
    $\mathcal{D}(\mathcal{N})^G$ maps $\phi$ to 0. Since $\mathcal{D}(\mathcal{N})^G$ is the
    dual space of $C_c^{\infty}(\mathcal{N})_G$, the space of coinvariants for $G$, $\phi$ also
    maps to $0$ under the projection
    $C_c^{\infty}(\mathcal{N}) \twoheadrightarrow C_c^{\infty}(\mathcal{N})_G$. Now, as in
    \cite{BZ76},
    \[
        \varinjlim_{0 \in V} C_c^{\infty}( \pi_G^{-1}V )_G \cong C_c^{\infty}(\mathcal{N})_G
    \]
    where the colimit runs over the compact open neighborhoods $V$ of 0 in $\mathbb{A}_G$. So if
    $\phi$ maps to $0$ in $C_c^{\infty}(\mathcal{N})_G$, then there must be a $V$ such that
    $\phi$ maps to $0$ in $C_c^{\infty}( \pi_G^{-1}V )_G$. Every distribution in
    $\mathcal{D}(\pi_G^{-1}V)^G$ therefore takes $\phi$ to 0. If we write $U = \pi^{-1}_GV$,
    then $O_X(\phi) = 0 \text{ for all } X \in U \cap \mathfrak{t}'$ or
    \begin{align*}
        O_X(f) & = \sum_{i = 1}^r \mu_i(f) \cdot \tilde{\Gamma}_i(X)
    \end{align*}
    as required.
\end{proof}

An important property of these germs is that they satisfy a partial homogeneity relation that
holds for squares of the field $F$.

\begin{prop}
    Let $d_i = \dim \mathcal{O}_i$. Then
    \begin{align}\label{shalika_homogeneity}
        \tilde{\Gamma}_i (\alpha^2X) & = | \alpha |^{-d_i} \tilde{\Gamma}_i (X)
    \end{align}
    for all $\alpha \in F^\times$.
\end{prop}

\begin{proof}
    Set $f_{\alpha^2} (X) = f(\alpha^2 X)$. Because
    $\mu_i ( f_{\alpha^2} ) = | \alpha |^{-d_i} \mu_i (f)$ (see e.g., \cite{HC70}), we can
    verify this identity directly on a small neighborhood of $0 \in \mathfrak{g}$. Explicitly,
    \begin{align*}
        \tilde\Gamma_i(\alpha^2 X) & = O_X( (f_i)_{\alpha^2} ) = \sum_j \mu_j ((f_i)_{\alpha^2}) \tilde\Gamma_j(X)
            = \sum_j |\alpha|^{-d_j} \mu_j(f_i) \tilde\Gamma_j(X) = | \alpha |^{-d_i} \tilde\Gamma_i(X)
    \end{align*}
    as required.
\end{proof}

We use this last proposition to define canonical representatives for these Shalika germs. Choose
a function $\Gamma_i'$ whose germ at the origin is $\tilde{\Gamma}_i$, as well as a lattice $L$
containing $0$ on which the homogeneity relation \eqref{shalika_homogeneity} holds for
$\alpha=\varpi$. For any $X \in \mathfrak{h}^{\perp}$, choose $c=\varpi^n$ with $n\geq0$
sufficiently large that $c^2X\in L$ and define
\begin{align*}
    \Gamma_i(X) & := | c |^{d_i} \Gamma'_i(c^2 X)
\end{align*}
We will usually refer to these canonical representatives as Shalika germs as well.

Orbital integrals also have Shalika germ expansions near possibly non-regular semisimple
elements. We will derive these from the Shalika germ expansions near the origin of the Lie
algebra $\mathfrak{l}$ of reductive subgroups $L$ in $G$ that arise as centralizers of
semisimple elements of $\mathfrak{g}$. Write $\tilde{\Gamma}_{\mathcal{O}}^L$ for the Shalika
germ on $\mathfrak{l}$ that corresponds to a given nilpotent $L$-orbit $\mathcal{O}$ in
$\mathfrak{l}$. Let $\Gamma_{\mathcal{O}}^L$ be the corresponding canonical representative,
which is defined on $\mathfrak{l}$. With this notation, for example, $\Gamma_i^G = \Gamma_i$.
Write $Z$ for the center of $G$ and $\mathfrak{z}$ for its Lie algebra. To derive these
generalized expansions, we will first establish central translation invariance directly.

\begin{lemma}
    For $Z\in\mathfrak z\cap\mathfrak h^\perp$, the canonical Shalika germs satisfy
    $\Gamma_i^G(X+Z)=\Gamma_i^G(X)$ on their semisimple domain.
\end{lemma}

\begin{proof}
    Choose a sufficiently small central lattice $L$ such that $f_i$ are translation invariant by
    elements of $L$. Since $X$ and $X+Z$ have the same centralizer, $O_{X+Z}(f_i)=O_X(f_i)$
    whenever $Z\in L$. For arbitrary $X,Z$, choose $c\in F^\times$ so small that $c^2Z\in L$ and
    both $c^2X$ and $c^2(X+Z)$ lie where the orbital integrals represent the germs. Then
    \[
        \Gamma_i^G(c^2(X+Z))
        =O_{c^2(X+Z)}(f_i)
        =O_{c^2X}(f_i)
        =\Gamma_i^G(c^2X)
    \]
    Square homogeneity completes the proof.
\end{proof}

This lemma is already enough to define a Shalika germ expansion near central semisimple elements
of $\mathfrak{g}$. At arbitrary semisimple elements $S$, the Shalika germ expansion near $S$ in
the centralizer $\mathfrak{g}_S$ provides a Shalika germ expansion near $S$ in $\mathfrak{g}$,
according to the next proposition.

\begin{prop}
    Let $S \in \mathfrak{h}^{\perp}$ be semisimple, and suppose that $T$ is a maximal
    $\theta$-split torus in $Z_G(S)$ with Lie algebra $\mathfrak{t}$. For any
    $f \in C_c^{\infty}(\mathfrak{g})$,
    \begin{align*}
        O_X(f) & = \sum_{i} \mu_{S + Y_i}(f) \cdot \Gamma_i^{Z_G(S)}(X)
    \end{align*}
    for all $X \in U \cap \mathfrak{t}'$, where $U$ is an open neighborhood of $S$ in
    $\mathfrak{g}_S$ and $\{ Y_i \}$ is a complete set of representatives for the nilpotent
    orbits of $Z_G(S)$ in $\mathfrak{g}_S$.
\end{prop}

\begin{proof}
    Put $L=Z_G(S)$. For $\phi\in C_c^\infty(\mathfrak l)$, the Shalika germ expansion at zero,
    applied to $\phi_S(Y)=\phi(S+Y)$, gives
    \[
        O_Y^L(\phi_S)
        =\sum_i\mu_{Y_i}^L(\phi_S)\Gamma_i^L(Y)
    \]
    whenever $Y$ is sufficiently close to zero and $\theta$-regular in $L$. Because $S$ is central in
    $L$, translation invariance by central elements implies that
    \[
        O_{S+Y}^L(\phi)
        =\sum_i\mu_{Y_i}^L(\phi_S)\Gamma_i^L(S+Y)
    \]
    By semisimple descent (\cite[\S16]{Kot06}), we can choose $\phi$ so that
    $O_X^G(f)=O_X^L(\phi)$ whenever $\pi_L(X)$ belongs to a sufficiently small neighborhood of
    $\pi_L(S)$ in $\mathbb A_L^{G\text{-reg}}$. With the normalization induced by descent,
    $\mu_{S+Y_i}(f)=\mu_{Y_i}^L(\phi_S)$, and the proposition follows.
\end{proof}

The Shalika germ expansions by themselves almost guarantee that orbital integrals near
semisimple elements become locally integrable when multiplied by $|D^G(X)|^{\frac{1}{2}}$,
except that certain germs might have degrees of homogeneity that are too large. Fortunately, we
can prove that the only terms that could possibly be problematic are identically zero.

\begin{definition}
    Let $\mathcal{O}(X)$ denote the $G$-orbit of $X \in \mathfrak{g}$. For any subset
    $S \subset \mathfrak{g}$, write $\mathcal{O}(S) := \bigcup_{X \in S} \mathcal{O}(X)$.
\end{definition}

\begin{prop}
    Let $T$ be some maximal $\theta$-split torus with Lie algebra $\mathfrak{t}$. Then
    \begin{align*}
        \dim \mathcal{O}_i > \dim Z_G(T) \backslash G \implies \Gamma_i(X) = 0 \text{ for all } X \in \mathfrak{t}'
    \end{align*}
\end{prop}

\begin{proof}
    Suppose that $\dim \mathcal{O}_i > \dim Z_G(T) \backslash G$. Let $\{ X_j : j \geq 0 \}$ be
    a convergent sequence in $\mathcal{O}(\mathfrak{t}')$ with limit $X$. Then for each
    $j \geq 0$,
    \begin{align*}
        \dim \mathcal{O}(X_j) = \dim Z_G(T) \backslash G < \dim \mathcal{O}_i
    \end{align*}
    Since the function $Y \mapsto \dim \mathcal{O}(Y) = \text{rank\,} \text{ad\,} Y$ is lower
    semicontinuous, we see that $X \notin \mathcal{O}_i$. Thus
    \begin{align*}
        \overline{\mathcal{O}(\mathfrak{t}')} \cap \mathcal{O}_i = \emptyset
    \end{align*}
    We claim that for any $Y \in \mathcal{O}_i$ and any open set $U$ containing $Y$ but disjoint
    from $\mathcal{O}_j$ with $j < i$, there is a function $f_i$ supported in $U$ such that
    $\Gamma_i(X)$ and $O_X(f_i)$ have the same germs at the origin.

    We prove this by induction. To begin, suppose that $i = r$. Then $\mathcal{O}_r$ is
    relatively open in $\mathcal{O}_1 \cup \cdots \cup \mathcal{O}_r$, so that
    $\mathcal{O}_r \cap U$ is relatively open in $\mathcal{O}_1 \cup \cdots \cup \mathcal{O}_r$
    and does not intersect $\mathcal{O}_1, \mathcal{O}_2, \cdots, \mathcal{O}_{r-1}$. Take a
    nonnegative locally constant compactly supported function $\tilde{f}_r$ in $U$ with
    $\tilde{f}_r(Y)>0$. Then $\mu_j(\tilde{f}_r) = 0$ whenever $j < r$ and
    $\mu_r(\tilde{f}_r) > 0$. The function
    \begin{align*}
        f_r & = \frac{\tilde{f}_r}{\mu_r(\tilde{f}_r)}
    \end{align*}
    has the sought properties.

    We now suppose that the claim is true for all $j > k$. $\mathcal{O}_k$ is relatively open in
    $\mathcal{O}_1 \cup \cdots \cup \mathcal{O}_k$, so that $\mathcal{O}_k \cap U$ is relatively
    open in $\mathcal{O}_1 \cup \cdots \cup \mathcal{O}_k$ and does not intersect
    $\mathcal{O}_1, \mathcal{O}_2, \cdots, \mathcal{O}_{k-1}$. Again, we take a nonnegative
    locally constant compactly supported function $\tilde{f}_k$ in $U$ with $\tilde{f}_k(Y)>0$.
    Set
    \begin{align*}
        f_k & := \frac{\tilde{f}_k}{\mu_k(\tilde{f}_k)} - \sum_{j > k} \frac{\mu_j(\tilde{f}_k)}{\mu_k(\tilde{f}_k)} f_j
    \end{align*}
    For each $j>k$ with $\mu_j(\tilde f_k)\ne0$, choose $Y_j\in U\cap\mathcal O_j$ and a
    neighborhood $U_j\subset U$ disjoint from the lower orbits. Apply the induction hypothesis
    inside $U_j$ to obtain a dual $f_j$ supported in $U_j$. Terms with zero coefficient are
    omitted. Thus every correction remains supported in $U$. For $h < k$,
    \begin{align*}
        \mu_h(f_k) & = \frac{\mu_h(\tilde{f}_k)}{\mu_k(\tilde{f}_k)} - \sum_{j > k}
            \frac{\mu_j(\tilde{f}_k)}{\mu_k(\tilde{f}_k)} \mu_h(f_j) =
            \frac{\mu_h(\tilde{f}_k)}{\mu_k(\tilde{f}_k)} = 0
    \end{align*}
    For $h > k$,
    \begin{align*}
        \mu_h(f_k) & = \frac{\mu_h(\tilde{f}_k)}{\mu_k(\tilde{f}_k)} - \sum_{j > k}
            \frac{\mu_j(\tilde{f}_k)}{\mu_k(\tilde{f}_k)} \mu_h(f_j) =
            \frac{\mu_h(\tilde{f}_k)}{\mu_k(\tilde{f}_k)} - \frac{\mu_h(\tilde{f}_k)}{\mu_k(\tilde{f}_k)} = 0
    \end{align*}
    For $h = k$,
    \begin{align*}
        \mu_h(f_k) & = \frac{\mu_h(\tilde{f}_k)}{\mu_k(\tilde{f}_k)} - \sum_{j > k}
            \frac{\mu_j(\tilde{f}_k)}{\mu_k(\tilde{f}_k)} \mu_h(f_j) =
            \frac{\mu_k(\tilde{f}_k)}{\mu_k(\tilde{f}_k)} = 1
    \end{align*}
    Therefore, this function has all the sought properties, and the claim follows inductively.

    Now take any $Y \in \mathcal{O}_i$ that does not belong to the closure of
    $\mathcal{O}(\mathfrak{t}')$, i.e., such that there is an open set $U$ containing $Y$ that
    does not intersect $\overline{\mathcal{O}(\mathfrak{t}')}$ or any lower nilpotent orbit. We
    can choose $f_i$ with support in this set, and define $\tilde{\Gamma}_i(X)$ to be the germ
    of the function
    \begin{align*}
        X \mapsto & \begin{cases}
                        \int_{Z_G(X) \backslash G} f_i(g^{-1}Xg) \; \dot{dg} & X \in
                            \mathfrak{h}^{\perp}_{\theta-reg}\text{ semisimple} \\
                        0                                                    & \text{otherwise}
                    \end{cases}
    \end{align*}
    This function is identically 0 on $\mathfrak t'$ because we are integrating outside the
    support of $f_i$. Thus $\tilde{\Gamma}_i$ vanishes on the germ of $\mathfrak t'$ at zero,
    and $\Gamma_i(X) = 0$ for $X\in\mathfrak t'$.
\end{proof}

This last result allows us to give the following lemma in full generality.

\begin{lemma}\label{normalized_orbital_boundedness}
    The function $\mathfrak{t} \rightarrow \mathbb{C}$ defined by extending
    \begin{align*}
        X & \mapsto |D^G_{Z_G(T)}(X)|^{\frac{1}{2}} \int_{Z_G(T) \backslash G} f(g^{-1}Xg) \; \dot{dg}
    \end{align*}
    by 0 from $\mathfrak{t}_{\theta-reg}$ to $\mathfrak{t}$ is locally bounded (and therefore
    locally integrable) on each Cartan subalgebra $\mathfrak{t}$ of $\mathfrak{h}^{\perp}$.
\end{lemma}

\begin{proof}
    The Shalika germ expansion of this integral at an element $Y$ of $\mathfrak{t}$ is
    \begin{align*}
        |D^G_{Z_G(T)}(X)|^{\frac{1}{2}} \int_{Z_G(T) \backslash G} & f(g^{-1}Xg) \; \dot{dg} \\
                                                                   & = \sum_{i} \mu_{Y+Y_i}(f) \cdot
                                                                       |D^G_{Z_G(T)}(X)|^{\frac{1}{2}} \cdot
                                                                       \Gamma_i^{Z_G(Y)}(X)
    \end{align*}
    It is enough to verify this lemma termwise, and check that each
    $|D^G_{Z_G(T)}(X)|^{\frac{1}{2}} \cdot \Gamma_i^{Z_G(Y)}(X)$ is locally bounded. Near $Y$,
    each term is (up to a constant)
    \begin{align*}
         & | D^{Z_G(Y)}_{Z_G(T)}(X)|^{\frac{1}{2}} \cdot \Gamma_i^{Z_G(Y)}(X)
    \end{align*}
    Because $Y$ belongs to the center of $Z_G(Y)$ and this function is translation invariant
    under central elements, it suffices to check boundedness after translating $Y$ to $0$. To do
    this, we consider the degrees of homogeneity for each factor. Write $L=Z_G(Y)$ and
    $\mathcal O_i^L$ for the $i^{th}$ nilpotent orbit in $L$. Then
    \begin{align*}
        \deg |D^L_{Z_G(T)}(X)|^{\frac12}
         & =\frac12(\dim L-\dim Z_G(T))         \\
        \deg\Gamma_i^L(X)
         & =-\frac12 \text{dim}\;\mathcal O_i^L
        \geq-\frac12(\dim L-\dim Z_G(T))
    \end{align*}
    or else $\Gamma_i^L$ vanishes on $\mathfrak t$. This means that
    \[
        \deg\left(
        |D^L_{Z_G(T)}(X)|^{\frac12} \cdot \Gamma_i^L(X)
        \right)\geq0
    \]    
    We can now prove local boundedness by induction on $\dim L$. First suppose that $S \in \mathfrak t$ is 
    noncentral in $\mathfrak l$. By homogeneity, we can replace $S$ by a small enough square 
    multiple so that the 
    canonical germs are represented by orbital integrals. The germ expansion then reduces the 
    question to \(Z_L(S)\), which has smaller dimension, so the induction hypothesis applies. 
    By central translation invariance, it remains to check boundedness near zero in the 
    complement $\mathfrak t_0$ of $\mathfrak t\cap\mathfrak z_L$. If this complement is zero, 
    we're done. Otherwise, choose a lattice $\Lambda\subset\mathfrak t_0$ and consider the 
    compact annulus $\Lambda\setminus\varpi^2\Lambda$. Since every point in this annulus is 
    noncentral, the local boundedness just established gives a uniform bound there. 
    Every nonzero point of $\Lambda$ is obtained from this annulus by repeatedly multiplying 
    by $\varpi^2$, which implies that this bound also holds near zero, completing the proof.
\end{proof}

\subsection{Bounding the Weighted Terms}

The crucial step that allows us to replace the initial weight factors $\omega$ by their
simplifications $\nu$ is an application of the Lebesgue dominated convergence theorem. This
requires several lemmas on the absolute summability of not just orbital integrals, but also of
their weighted analogs. These lemmas are perhaps easiest to approach using the language of
abstract norms on $F$-varieties. We will state some results in this section without proof. A
detailed exposition of the missing arguments is \cite{Kot06}, covering material due to
Harish-Chandra, which we follow closely.

An abstract norm on an $F$-variety is just a function $\| \cdot \|$ whose value is always
greater than or equal to 1. Among abstract norms, there is a notion of equivalence: one abstract
norm is equivalent to another if the first is bounded by a constant times a power of the second
and vice versa. We will only be interested in a certain equivalence class of abstract norms.

Namely, if $U$ is an affine scheme whose algebra of global sections $\mathcal{O}_U$ is generated
by the functions $f_1$, $f_2$, \ldots, $f_n$, then we can define an abstract norm as follows:
\[
    \| x \|_U := \text{max\,} \left\{ 1, |f_1(x)|, |f_2(x)|, \ldots, |f_n(x)| \right\}
\]
where $x \in U$. For schemes $X$ of finite type over $F$, we choose a collection of affine open
subsets of $X$ that cover $X$, e.g., $U_1$, $U_2$, $\ldots$, $U_{n-1}$, and $U_n$. Define the
norm
\[
    \| x \|_X := \text{inf\,} \left\{ \| x \|_{U_i} : x \in U_i \right\}
\]
for all $x \in X$. While this abstract norm depends on the choice of the affine sets $U_i$ as
well as the generating functions $f_i$ on each, the equivalence class it defines is independent
of all these choices, and so we define a norm on a scheme of finite type over $F$ to be an
abstract norm that belongs to this equivalence class. An example of such a norm that we have
already encountered is the function $\|\cdot\|_G:=\text{exp\,}d(\cdot)$ on the algebraic group
$G$. The next proposition is explained in \cite{Kot06} and lists some of the more elementary
properties of norms.

\begin{prop}
    \label{norm_basics}
    Let $X$ and $Y$ be affine schemes of finite type over $F$ and let $\| \cdot \|_X$ and
    $\| \cdot \|_Y$ be norms on $X(F)$ and $Y(F)$ respectively.
    \begin{itemize}
        \item Let $\phi : Y \rightarrow X$ be a morphism and denote by $\phi^{*} \| \cdot \|_X$
            the abstract norm on $Y(F)$ obtained by composing $\| \cdot \|_X$ with
            $\phi : Y(F) \rightarrow X(F)$. Then $\| \cdot \|_Y$ dominates
            $\phi^{*} \| \cdot \|_X$. If $\phi$ is finite, then $\| \cdot \|_Y$ is equivalent to
            $\phi^{*} \| \cdot \|_X$.
        \item Suppose $Y$ is a closed subscheme of $X$. Then the restriction of $\| \cdot \|_X$
            to $Y(F)$ is equivalent to $\| \cdot \|_Y$.
        \item If $F$ is locally compact, then a subset $B(F)$ of $X(F)$ has compact closure if
            and only if it is bounded, i.e., if the norm function $\| \cdot \|_X$ is bounded on
            $B(F)$.
        \item All three of $\text{sup\,} \{ \| x \|_X, \| y \|_Y \}$, $\|x\|_X + \|y\|_Y$, and
            $\|x\|_X \cdot \|y\|_Y$ are valid norms on $(X \times Y)(F) = X(F) \times Y(F)$.
        \item Let $U := X_f$ denote the principal open subset of $X$ determined by a regular
            function $f$ on $X$, so that $U(F) = \{ x \in X(F) : f(x) \neq 0 \}$. Then
            $\| u \|_U := \text{sup\,} \{ \| u \|_X, | f(u) |^{-1} \}$ is a norm on $U(F)$.
        \item Suppose we are given a finite cover of $X$ by affine open subsets $U_1$, $U_2$,
            $\ldots$, $U_r$ as well as a norm $\| \cdot \|_i$ on $U_i(F)$ for each
            $i = 1, 2, \ldots, r$. For $x \in X(F)$, define $\| x \|$ to be the infimum of the
            numbers $\| x \|_i$, where $i$ ranges over the set of indices for which
            $x \in U_i(F)$. Then $\| \cdot \|$ is a norm on $X(F)$.
        \item Let $G$ be a group scheme of finite type over $F$, and suppose we are given an
            action of $G$ on $X$. Let $B$ be a bounded subset of $G(F)$. Then there exist
            $c, R > 0$ such that $\| bx \|_X \leq c \| x \|_X^R$ for all $b \in B, x \in X(F)$.
    \end{itemize}
\end{prop}

If we are given a morphism of schemes $\varphi: X \rightarrow Y$ and a norm $\| \cdot \|_X$ on
$X$, we can define an abstract norm on the image of $X(F)$ in $Y(F)$:
\[
    \| y \|_Y := \text{inf\,} \left\{ \| x \|_X : \varphi(x) = y \right\}
\]
We will call this norm the push-forward of $\| \cdot \|_X$ and sometimes denote it
$\varphi_* \| \cdot \|_X$. When this abstract norm is a norm on the image of $\varphi$, it will
be a norm for all choices of $\| \cdot \|_X$, and we say that the morphism $\varphi$ has the
norm descent property. The behavior of the norm descent property under composition of morphisms
is described by the next lemma.

\begin{lemma}
    Consider morphisms $f : X \rightarrow Y$ and $g : Y \rightarrow Z$ of affine schemes of
    finite type over $F$. Put $h = g \circ f : X \rightarrow Z$. Assume that the map
    $f : X(F) \rightarrow Y(F)$ is surjective. Then
    \begin{itemize}
        \item If $f$ and $g$ satisfy the norm descent property, then so does $h$.
        \item If $h$ satisfies the norm descent property, then so does $g$.
    \end{itemize}
\end{lemma}

A sufficient condition for a morphism to have the norm descent property is for it to admit
sections on its image locally in the Zariski topology. For example, Bruhat theory defines an
open immersion associated to each Levi subgroup $M$,
\[
    \bar{U} \times M \times U \hookrightarrow G
\]
which is given by multiplication. The image of $\bar{U} \times \{ 1 \} \times U$ projects to an
open subset in $M \backslash G$, and defines a section over this open subset. By translating
this ($G$-equivariant) morphism by elements of $G$, we obtain sections over the open sets in
some covering of $M \backslash G$, which implies that the morphism
\[
    G \rightarrow M \backslash G
\]
has the norm descent property. This classical result implies that
\begin{align*}
    \| g \|_{M \backslash G} & := \text{inf\,} \{ \| mg \|_G : m \in M \}
\end{align*}
is a norm on the affine $F$-variety $M \backslash G$.

Other morphisms with the norm descent property include finite morphisms and quotients of the
form $G \rightarrow T \backslash G$, where $T$ is any torus of $G$. In other words,
\begin{align*}
    \| g \|_{T \backslash G} & := \text{inf\,} \{ \| tg \|_G : t \in T \}
\end{align*}
is a norm on $T \backslash G$. Especially important is the case in which $T$ is the $F$-split
torus $A_M$.

For $M\in\mathcal L^-$ and $T\in\mathcal T_M$, the projection
\[
    p_{M,T}:(A_M^-\cdot T[2])_F\backslash G_F\longrightarrow Z_G(T)_F\backslash G_F
\]
is proper and has compact fiber $(A_M^-\cdot T[2])_F\backslash Z_G(T)_F$. Fix a norm on
$Z_G(T)_F\backslash G_F$ and define
\[
    \sigma_{M,T}(g) := \sigma_{M,T}((A_M^-\cdot T[2])_Fg):=\|Z_G(T)_Fg\|_{Z_G(T)_F\backslash G_F}
\]
In the next lemma, we will use these norms to bound the weight factors.

\begin{lemma}
    \label{wf_bounds}
    For every $M\in\mathcal L^-$ and $T\in\mathcal T_M$,
    \begin{align*}
        |\omega_{M,T} (g, \mu)| & \leq
        a ( 1 + \log \sigma_{M,T}(g) + \| \mu \|_E )^r \\
        |\nu_{M,T}(g, \mu)|     & \leq
        a ( 1 + \log \sigma_{M,T}(g) + \| \mu \|_E )^r
    \end{align*}
    for some constants $a$ and $r$.
\end{lemma}

\begin{proof}
    Set
    \[
        C_{M,T}:=\sum_{S\in\mathcal T_G(T)}
        \frac{|W_G(S)|}{|W_H(S)|} C_{M,T,S} = \sum_{S\in\mathcal T_G(T)}
        \frac{|W_G(S)|}{|W_H(S)|}  \operatorname{vol}_{ds} Q_{M,T}(\lambda)
    \]
    We can find a compact subset of $Z_G(T)_F$ that projects surjectively onto
    $(A_M^-\cdot T[2])_F\backslash Z_G(T)_F$. The boundedness of this compact set, with norm
    descent for $G_F\to Z_G(T)_F\backslash G_F$, shows that every point of
    $(A_M^-\cdot T[2])_F\backslash G_F$ lifts to some $g_0 \in G_F$ satisfying
    \[
        1+d(g_0)+d(\theta(g_0))
        \leq C\bigl(1+\log\sigma_{M,T}(g)\bigr)
    \]
    maybe after changing $C$.

    On the truncation function's support, the triangle inequality in the Bruhat-Tits building
    implies that
    \[
        \|\nu_{s^2}\|_E\leq
        C\bigl(1+d(g_0)+d((\theta(g_S)^{-1}g_S)^{-1}\theta(g_0))+\|\mu\|_E\bigr)
    \]
    The integrand in $\omega_{M,T}$'s definition vanishes when this constraint fails. For each
    $S$, then, we can bound $\omega_{M,T}$ by $C_{M,T,S}$ times the (finite) number of
    lattice points in this region. Because there are finitely many $S$, the inequality for
    $\omega_{M,T}$ follows.

    Next we consider $\nu_{M,T}$. In subsection~\ref{polynomial_wf}, we saw that each
    \[
        EP\bigl(z,\mathcal L_{M,T,S}(g,\mu)\bigr)
    \]
    is a finite sum of polynomials times finite-order characters. These characters have absolute
    value one, while the polynomials depend on $\mu$, $H_P(g_0)$, and
    $H_{\bar P}((\theta(g_S)^{-1}g_S)^{-1}\theta(g_0))$ (after absorbing the translations
    $a_{M,T,S}(g,\mu)$ into the character factors). Because the Iwasawa terms are bounded by
    $C(1+\log\sigma_{M,T}(g))$, giving a total bound of $C'(1+\log\sigma_{M,T}(g)+\|\mu\|_E)$,
    we get
    \[
        \left|EP\bigl(z,\mathcal L_{M,T,S}(g,\mu)\bigr)\right|
        \leq C'\bigl(1+\log\sigma_{M,T}(g)+\|\mu\|_E\bigr)^{r'}
    \]
    uniformly in $S$ and $z$, both ranging over finite sets. The inequality for $\nu_{M,T}$
    follows from summing the finitely many terms in its definition.
\end{proof}

\begin{lemma}
    \label{branch_support}
    Let $\mu_1$ be dominant and let $\mu_2\in\Lambda^0$ be strictly dominant. There is a
    constant $C$ such that, for every $d\in\mathbb Z_{\geq0}$,
    \begin{align*}
         & \int_{Q_{M,T}}\bar\omega(g_Ssg,\mu_1+d\mu_2)\,ds
            ={}
        \frac{C_{M,T,S}}{|\mathcal Z_{M,T,S}|}
        \sum_{z\in\mathcal Z_{M,T,S}}
        EP\bigl(z,\mathcal L_{M,T,S}(g,\mu_1+d\mu_2)\bigr)
    \end{align*}
    whenever $d\geq C\bigl(1+\log\sigma_{M,T}(g)\bigr)$. The constant can be chosen uniformly.
\end{lemma}

\begin{proof}
    We first rule out two trivial cases: if
    $\xi_{M,T,S}(g,\mu_1+d\mu_2)\notin\operatorname{im}\kappa_{M,T}$, then both sides vanish; if
    $\Delta$ is empty, then $G$ is a torus and both sides equal $C_{M,T,S}$.

    Otherwise, because the integral (and its quasipolynomial continuation) is left
    $(A_M^-\cdot T[2])_F$-invariant, norm descent and the compactness of
    $(A_M^-\cdot T[2])_F\backslash Z_G(T)_F$ let us replace $g$ by a $g_0$ satisfying
    \[
        1+d(g_0)+d(g_S^{-1}\theta(g_Sg_0))
        \leq C_0(1+\log\sigma_{M,T}(g))
    \]
    uniformly in $S$.

    Because $\mu_1$ is dominant and $\mu_2$ is strictly dominant,
    \[
        \langle\alpha,\mu_1+d\mu_2\rangle
        \geq d\min_{\beta\in\Delta}\langle\beta,\mu_2\rangle
        \qquad(\alpha\in\Delta)
    \]
    For big enough $C$ then, we can apply Arthur's key geometric lemma at $g_0$ whenever
    $d\geq C(1+\log\sigma_{M,T}(g))$. The corresponding orthogonal set then has to be special,
    and our lattice-counting corollary implies that
    \begin{align*}
        \int_{Q_{M,T}}\bar\omega(g_Ssg_0,\mu_1+d\mu_2)\,ds
         & =C_{M,T,S}\,
        \bigl|P^-_{M,T,S}(g_0,\mu_1+d\mu_2)\bigr|  \\
         & =\frac{C_{M,T,S}}{|\mathcal Z_{M,T,S}|}
        \sum_{z\in\mathcal Z_{M,T,S}}
        EP\bigl(z,\mathcal L_{M,T,S}(g_0,\mu_1+d\mu_2)\bigr)
    \end{align*}
    Since $g_0$ and $g$ differ by an element of $(A_M^-\cdot T[2])_F$, this establishes the
    equality at $g$ as well. Since all our indexing sets are finite, we can use the same $C$ for
    every term.
\end{proof}

\begin{lemma}
    \label{quotient_convergence}
    For every $f\in C_c^\infty(\mathfrak g)$ and integer $r\geq0$, the integral
    \begin{align*}
        \int_{\mathfrak t'}|D^G_{Z_G(T)}(X)|^{\frac12}
        \int_{(A_M^-\cdot T[2])_F\backslash G_F}
        |f(g^{-1}Xg)|
        \bigl(1+\log\sigma_{M,T}(g)\bigr)^r\,dg\,dX
    \end{align*}
    converges.
\end{lemma}

\begin{proof}
    For $X\in\mathfrak t'$, one has $Z_G(X)=Z_G(T)$ and hence $D^G(X)=D^G_{Z_G(T)}(X)$. In
    particular, $g\mapsto |f(g^{-1}Xg)|$ is left $Z_G(T)_F$-invariant. Integration by stages and
    the definition of $\sigma_{M,T}$ give
    \begin{align*}
         & \int_{(A_M^-\cdot T[2])_F\backslash G_F}|f(g^{-1}Xg)|
        \bigl(1+\log\sigma_{M,T}(g)\bigr)^r\,dg                  \\
         & \qquad={}
        \operatorname{vol}((A_M^-\cdot T[2])_F\backslash Z_G(T)_F)
        \int_{Z_G(T)_F\backslash G_F}|f(g^{-1}Xg)|
        \bigl(1+\log\|Z_G(T)_Fg\|_{Z_G(T)_F\backslash G_F}\bigr)^r\,dg
    \end{align*}
    So we can just prove convergence of the latter integral, on $Z_G(T)_F\backslash G_F$.

    Choose a norm on $\mathfrak g$, and write
    \[
        \|Y\|_{rs}:=\max\left\{\|Y\|_{\mathfrak g},
        |D^G(Y)|^{-1}\right\}
    \]
    The map
    \[
        \beta:\mathfrak t_{\rm reg}\times(T\backslash G)\longrightarrow
        \mathfrak g_{\rm rs}\qquad (X,Tg)\longmapsto g^{-1}Xg
    \]
    is finite with fibers indexed by the Weyl group. The pullback of $\| \cdot \|_{rs}$ is
    therefore a norm and dominates the quotient norm after projecting to $T\backslash G$.
    Proposition~\ref{norm_basics} then provides constants $C,N$, independent of $X$, as in
    \cite[Lemma~20.3]{Kot06}, such that
    \[
        \|Z_G(T)_Fg\|_{Z_G(T)_F\backslash G_F}
        \leq C\|g^{-1}Xg\|_{rs}^N
    \]
    whenever $f(g^{-1}Xg)\neq0$. On the support of $f$, the first term in the definition of
    $\|g^{-1}Xg\|_{rs}$ is bounded, while the discriminant is invariant under conjugation.

    For a fixed $r$, $(1+\log t)^r\ll t^\epsilon$ for $t\geq1$ and every $\epsilon>0$. So the
    preceding norm estimate gives
    \[
        \bigl(1+\log\|Z_G(T)_Fg\|_{Z_G(T)_F\backslash G_F}\bigr)^r
        \leq C_\epsilon |D^G_{Z_G(T)}(X)|^{-\epsilon}
    \]
    on the support of the integrand. This means that the normalized orbital integral
    \[
        |D^G_{Z_G(T)}(X)|^{\frac12}
        \int_{Z_G(T)_F\backslash G_F}|f(g^{-1}Xg)|\,dg
    \]
    is locally bounded by the preceding orbital-integral lemma. Its support in $\mathfrak t$ is
    compact because $\mathfrak t\to\mathbb A_G$ is finite and $\pi_G(\operatorname{supp}f)$ is
    compact, and $|D^G_{Z_G(T)}(X)|^{-\epsilon}$ is locally integrable for sufficiently small
    $\epsilon$. The lemma follows.
\end{proof}

\section{\texorpdfstring{The $\theta$-split Side of the Relative Local Trace Formula}{The theta-split Side of the Relative Local Trace Formula}} \label{5}

With the results of section \hyperref[4]{4} in hand, we can prove the limit of
Proposition~\ref{main_result}, which is the goal of this paper.
\begin{prop}[Split-side approximation]\label{split_approximation}
    Let $\mu_1$ be a dominant coweight and let $\mu_2\in\Lambda^0$ be a strictly dominant
    coweight. For $d\in\mathbb Z_{\geq0}$, set $\mu:=\mu_1+d\mu_2$. Then
    \begin{align*}
         & \lim_{\substack{d\rightarrow\infty \\ d\in\mathbb Z}}
        \bigl(J_-(f,\omega_\mu)-J_-(f,\nu_\mu)\bigr)=0
    \end{align*}
    The convergence is uniform on the supremum-norm unit ball of every finite-dimensional test
    space.
\end{prop}

\begin{proof}
    Fix $M$ and $T$, and write $\omega=\omega_{M,T}$, $\nu=\nu_{M,T}$, and
    $\sigma=\sigma_{M,T}$. By Lemma~\ref{branch_support},
    \[
        \omega(g,\mu)=\nu(g,\mu)
        \qquad\text{whenever}\qquad
        d\geq C(1+\log\sigma(g))
    \]
    In particular, their difference tends to zero.

    To take the limit under each integral, we need a bound independent of $d$. Since
    $\|\mu_1+d\mu_2\|_E$ is bounded by a constant multiple of $1+d$, Lemma~\ref{wf_bounds}
    implies that
    \[
        |\omega(g,\mu)-\nu(g,\mu)|
        \leq a(1+\log\sigma(g)+d)^r
    \]
    But this difference vanishes unless $d<C(1+\log\sigma(g))$. So after increasing $a$,
    \[
        |\omega(g,\mu)-\nu(g,\mu)|
        \leq a(1+\log\sigma(g))^r
    \]
    for every $d$. By Lemma~\ref{quotient_convergence}, this bound is integrable against the
    absolute value of the normalized orbital integrand. By the dominated convergence theorem
    then, the required limit holds for each $M,T$, and summing over these finitely many pairs
    proves the first assertion.

    Last, choose a basis of any fixed finite-dimensional test space. The coordinates of
    functions in its supremum-norm unit ball are uniformly bounded, so linearity and convergence
    on this finite basis imply uniform convergence on that ball, and we're done.
\end{proof}

Notice that the character expansion of these weights can be integrated coefficientwise by
Lemmas~\ref{wf_bounds} and \ref{quotient_convergence}, giving polynomials $P_{-,\chi}$
satisfying
\[ 
    J_-(f,\nu_\mu)
    =
    \sum_{\chi\in\widehat{\Lambda/\Lambda^0}}
    \chi(\mu)P_{-,\chi}(f;\mu)
    \qquad
    \deg P_{-,\chi}\leq\operatorname{ssrank}G
\]
where each $P_{-,\chi}$ has distribution coefficients.

\end{document}